\documentclass{amsart}

\usepackage{amsmath,xypic,leftidx}
\usepackage{xspace}
\usepackage[psamsfonts]{amssymb}
\usepackage[latin1]{inputenc}
\usepackage{graphicx,color}

\DeclareFontFamily{OT1}{pzc}{}
\DeclareFontShape{OT1}{pzc}{m}{it}{<-> s * [1.100] pzcmi7t}{}
\DeclareMathAlphabet{\mathpzc}{OT1}{pzc}{m}{it}

\usepackage{hyperref}

\theoremstyle{remark}

\newtheorem*{question}{\bf Question}

\newtheorem{property}{\bf Property}

\theoremstyle{plain}

\newtheorem{proposition}{\bf Proposition}[section]
\newtheorem{mproposition}{\bf Proposition}

\newtheorem{mexample}{\bf Example}

\newtheorem{definition}[proposition]{\bf Definition}
\newtheorem*{theorem}{Theorem}
\newtheorem*{maintheorem}{\bf  Main Theorem}
\newtheorem*{maintheoremprecised}{\bf Main Theorem (precise form)}

\newtheorem{lemma}[proposition]{\bf Lemma}
\newtheorem{corollary}[proposition]{\bf Corollary}

\def\eps{{\varepsilon}}

\def\C{{\mathbb C}}
\def\R{{\mathbb R}}

\def\H{{\mathbb H}}

\def\P{{{\mathbb P}}}

\def\N{{\mathbb N}}
\def\D{{\mathbb D}}

\def\F{{\cal F}}
\def\E{{\cal E}}

\def\cal{\mathcal}

\def\and{{\quad \text{and}\quad }}
\def\with{{\quad \text{with} \quad }}
\def\as{{\quad \text{as}\quad}}
\def\whence{{\quad \text{whence}\quad}}
\def\sothat{{\quad \text{so that}\quad}}
\def\whereas{{\quad \text{whereas}\quad}}
\def\where{{\quad \text{where}\quad}}
\def\for{{\quad \text{for}\quad}}
\def\forsome{{\quad \text{for some}\quad}}

\def\ds{\displaystyle}

\def\i{{\rm i}}
\def\e{{\rm e}}

\def\frak{\mathfrak }
\def\o{{\rm o}}
\def\O{{\rm O}}

\def\f{\boldsymbol{\mathsf{f}}}

\def\att{{\rm att}}
\def\rep{{\rm rep}}

\newcommand{\mz}{\boldsymbol{\mathpzc{Z}}}

\newcommand{\mx}{\boldsymbol{\mathpzc{X}}}
\newcommand{\mf}{\boldsymbol{\mathpzc{F}}}
\newcommand{\mR}{\boldsymbol{\mathpzc{R}}}
\newcommand{\mQ}{\boldsymbol{\mathpzc{Q}}}
\newcommand{\mS}{\boldsymbol{\mathpzc{S}}}
\newcommand{\mH}{\boldsymbol{\mathpzc{H}}}
\newcommand{\mD}{\boldsymbol{\mathpzc{D}}}
\newcommand{\mmu}{\boldsymbol{\mathpzc{u}}}
\newcommand{\mv}{\boldsymbol{\mathpzc{v}}}

\newcommand{\xniota}{x_n^\iota}
\newcommand{\xno}{x_n^{o}}
\newcommand{\xo}{x^{o}}
\newcommand{\uniota}{v_n^\iota}
\newcommand{\uno}{v_n^{o}}

\def\id{{\rm id}}

\def\Skew{{P}}

\def\fL{{\frak L}}
\def\L{{\cal L}}
\def\Pa{{  P}^\att_f}
\def\Pr{{  P}^\rep_f}
\def\Pg{B_r}

\def\pr{{\pi_1}}

\def\Chat{{\widehat \C}}
\newcommand{\bpf}{\mathcal{B}_f}
\newcommand{\bpg}{\mathcal{B}_g}

 \newcommand{\cv}{\rightarrow}

\newcommand{\set}[1]{\left\{#1\right\}}

\newcommand{\abs}[1]{\left\vert#1\right\vert}

\newcommand{\rest}[1]{ \arrowvert_{#1}}

\newcommand{\unsur}[1]{\frac{1}{#1}}

\newcommand{\lrpar}[1]{\left(#1\right)}

\date{\today}

\title[A polynomial map with a wandering Fatou component]{A two-dimensional polynomial mapping  with a wandering Fatou component}

\begin{author}[M.~Astorg]{Matthieu Astorg}
\address{ %
 Universit\'e Paul Sabatier\\
Laboratoire Emile Picard \\
 118, route de Narbonne \\
 31062 Toulouse Cedex \\
  France }
  \email{mastorg@math.univ-toulouse.fr }
\end{author}

\begin{author}[X.~Buff]{Xavier Buff}
\address{ %
 Universit\'e Paul Sabatier\\
Laboratoire Emile Picard \\
 118, route de Narbonne \\
 31062 Toulouse Cedex \\
  France }
  \email{xavier.buff$@$math.univ-toulouse.fr}
\end{author}

\begin{author}[R.~Dujardin]{Romain Dujardin}

\address{ %
Laboratoire d'Analyse et de Math\'ematiques Appliqu\'ees\\
Universit\'e de Marne-la-Vall\'ee\\
5, boulevard Descartes\\
Cit\'e Descartes  \\
77454 Champs-sur-Marne\\
France }
\email{romain.dujardin@u-pem.fr}
\end{author}

\begin{author}[H.~Peters]{Han Peters}
\address{ %
KdV Institute for Mathematics \\ University of Amsterdam\\ The Netherlands}
\email{h.peters@uva.nl}
\end{author}

\begin{author}[J.~Raissy]{Jasmin Raissy}

\address{ %
 Universit\'e Paul Sabatier\\
Laboratoire Emile Picard \\
 118, route de Narbonne \\
 31062 Toulouse Cedex \\
  France }
  \email{jraissy@math.univ-toulouse.fr}
\end{author}

\thanks{The work of  Matthieu Astorg, Xavier Buff,  Romain Dujardin and Jasmin Raissy was partially supported by ANR
 project LAMBDA,  ANR-13-BS01-0002. The work of Jasmin Raissy was also partially supported by
 the FIRB2012 grant ``Differential Geometry and Geometric Function Theory'', RBFR12W1AQ 002. }

\begin{document}

\begin{abstract}
We show that there exist polynomial endomorphisms of $\C^2$, possessing  a wandering Fatou component.
 These mappings are polynomial skew-products, and
 can be chosen to  extend   holomorphically of $\P^2(\C)$.  We also find real examples
  with wandering domains in $\mathbb{R}^2$. The proof relies on parabolic implosion techniques,
  and is based on an original idea of M. Lyubich.
\end{abstract}

\maketitle

\section*{Introduction}

Let $\Skew:\C^k\cv \C^k$ be a polynomial mapping. In this article we consider $\Skew$ as a dynamical system, that is, we study
 the behavior of the sequence of iterates $(\Skew^{\circ n})_{n\in \N}$. A case of particular interest is when $\Skew$ extends as a
 holomorphic endomorphism on $\P^k(\C)$. As a matter of expositional simplicity, let us assume for the moment  that this property holds.
  The {\em Fatou set} $\F_\Skew$ is classically defined as the
 largest open subset of $\P^k(\C)$ in which the sequence of iterates is locally equicontinuous (or {\em normal}, according to the usual terminology). Its complement, the {\em Julia set}, is where chaotic dynamics takes place. A {\em Fatou component} is a connected component of $\F_\Skew$.

 \medskip

When the dimension  $k$ equals 1, the Non Wandering Domain Theorem due to   Sullivan \cite{sul}
asserts that every Fatou component is  eventually periodic.
 This result is fundamental for at least two reasons.
 \begin{itemize}
 \item[-]  First, it opens the way to a complete description  of the  dynamics in the Fatou set: the orbit of any point in the Fatou set
 eventually lands in a (super-)attracting basin, a parabolic basin or a rotation domain.
 \item[-] Secondly, it introduced   quasi-conformal mappings as a new tool in this research area, leading to many subsequent developments.
 \end{itemize}

There are many variants and generalizations of Sullivan's Theorem in several areas of
one-dimensional dynamics.
For instance it was shown by   Eremenko and  Lyubich \cite{eremenko lyubich} and   Goldberg and   Keen \cite{gk}
that entire mappings with finitely many singular values have no wandering domains.
On the other hand,    Baker \cite{baker},
prior to Sullivan's result, gave the first example of an entire mapping with a wandering domain. 
Simple explicit entire mappings with  Fatou components  wandering to infinity were 
given in \cite[\S 9]{sul} and \cite[\S II.11]{herman}, while 
more elaborate examples with varied dynamical behaviors were presented in \cite{el1}.
More recently,   Bishop \cite{bishop} constructed an example with a bounded singular set. In
all cases, the orbit of the wandering domain is unbounded.

In the real setting, the question of   (non-)existence of wandering intervals has a long and rich history.   It started 
with Denjoy's theory of linearization of circle diffeomorphisms \cite{denjoy}: a $C^2$-diffeomorphism of the circle with irrational rotation 
number has no wandering intervals (hence it is linearizable), whereas this result breaks down for  $C^1$ diffeomorphisms. More recent results include homeomorphisms with various degrees of regularity and flatness of critical points. 

For interval maps, the non-existence of wandering intervals for unimodal maps with negative Schwarzian was established by 
Guckenheimer \cite{guckenheimer}, and later extended to 
several classes of interval and circle
 maps in \cite{lyubich, blokh lyubich, martens melo strien}. In particular,
the result of   Martens,     de Melo and   van Strien implies the
non-existence of wandering intervals for polynomials on the real line.
%

One difficulty  is to define a notion of  Fatou set   in this context. 
Let us just note here that for a real polynomial,
the Fatou set as defined in \cite{martens melo strien}
contains the  intersection of the complex Fatou set with the real line but could {\em a priori} be larger.

The problem  was also studied in the  non-Archimedian setting, in particular in
 the work of   Benedetto \cite{benedetto} and   Trucco \cite{trucco}.

 \medskip \begin{center}$\diamond$
\end{center}
\medskip

The question of the existence of wandering Fatou components makes sense in higher dimension, and was put forward by many authors
since the 1990's (see e.g. \cite{fs}). Higher dimensional transcendental mappings with wandering domains
can be constructed from one-dimensional examples by simply taking products.
 An example of a  transcendental biholomorphic map in $\C^2$ with   a wandering Fatou component oscillating to infinity was constructed
 by Forn\ae ss and   Sibony   in \cite{fs transcendental}.

For  higher dimensional polynomials and rational mappings, it is widely acknowledged that quasi-conformal methods break down,
so a direct approach to generalize Sullivan's Theorem  fails.
Besides this observation, little was known about the problem so far.

Recently, M. Lyubich suggested that  polynomial skew products
 with wandering domains might be constructed by using parabolic implosion techniques.
 The idea was to combine slow convergence  to an invariant fiber and parabolic transition in the fiber direction, to produce
 orbits shadowing those of a Lavaurs map (see below for a more precise description).

In this paper, we bring this idea to completion,
thereby providing the first examples of higher dimensional polynomial mappings
with wandering domains.

\begin{maintheorem}
There exists an endomorphism $\Skew:\P^2(\C)\to \P^2(\C)$, induced by  a polynomial skew-product mapping
$\Skew:\C^2\to \C^2$,
possessing a wandering Fatou component.
\end{maintheorem}

Let us point out that the orbits in these wandering domains are bounded.
The 
approach is in fact essentially local.  A more detailed statement is the following (see below for the definition of Lavaurs maps).

\begin{maintheoremprecised}
Let $f:\C\to \C$ and $g:\C\to \C$ be polynomials of the form
\begin{equation}\label{eq:fg}
f(z) = z+z^2+\O(z^3)\and g(w) = w-w^2+\O(w^3).
\end{equation}
If  the Lavaurs map $\L_f:\bpf\to \C$ has an attracting fixed point, then the skew-product map
 $\Skew:\C^2\to \C^2$ defined by
\begin{equation}\label{eq:skew}
\Skew(z,w) := \left(f(z)+\frac{\pi^2}{4} w,g(w)\right)
\end{equation}
admits a wandering Fatou component.
\end{maintheoremprecised}


Notice that if $f$ and $g$ have the same degree, $\Skew$ extends to a holomorphic self-map on $\P^2(\C)$.
Observe also that examples in arbitrary dimension $k\geq 2$ may   be obtained from this result
by simply considering   products mappings of the form  $(\Skew,Q)$,
where $Q$ admits  a fixed Fatou component.

\medskip

Before explaining what the Lavaurs map is, let us give some explicit examples satisfying the assumption of the Main Theorem.

\begin{mexample}\label{ex:complex}
Let  $f:\C\to \C$ be the cubic  polynomial
$f(z) = z+z^2+az^3$, and   $g$ be as in \eqref{eq:fg}. If  $r>0$ is sufficiently small and $a\in D(1-r, r)$, then
the polynomial skew-product $\Skew$ defined in \eqref{eq:skew} admits a wandering Fatou component.
\end{mexample}

Numerical experiments  suggest that the value $a=0.95$ works (see Figure \ref{fig:0.95} on page \pageref{fig:0.95}).

\medskip

In view of the results of \cite{martens melo strien} cited above,
it is also of interest to look for real polynomial mappings with wandering Fatou domains intersecting $\R^2$. Our method also provides
such examples.

\begin{mexample}\label{ex:real}
Let $f:\C\to \C$ be the degree $4$ polynomial  defined by
\[f(z):=z+z^2+bz^4\with b\in \R.\]
There exist parameters $b\in (-8/27,0)$ such that if $g$ is as in \eqref{eq:fg},  the
polynomial skew-product $\Skew$ defined in \eqref{eq:skew}
admits a wandering Fatou component intersecting~$\R^2$.
\end{mexample}

Numerical experiments  suggest that the  parameter $b=-0.2136$ satisfies this property.
We illustrate this phenomenon graphically in Figure  \ref{realwandering} for the mapping~$\Skew$ defined by
\begin{equation}\label{eq:defFreal}
(z,w)\longmapsto \left(z+z^2-0.2136 z^4 +\frac{\pi^2}{4}w, w-w^2\right)
\end{equation}
The set of points $(z,w)\in \R^2$ with bounded orbit is contained in the rectangle $(-3, 3)\times (0,1)$.
The topmost image in
Figure \ref{realwandering} displays  this set of points.
 The bottom right image is a zoom near the point $(z=-0.586, w=1/(1000^2))$. A wandering component  is visible
 (in green). The bottom left image displays the window $-3<z<3$ and $1/1003^2<w<1/999^2$;
 the orbit of a point $(z_0,w_0)$ contained in the wandering domain is indicated. The coordinate $w_0$ is close to
 $1/1000^2$ and we plotted the first $2001$ and the next $2003$ iterates.
 The points are indicated in black except $(z_0,w_0)$, $\Skew^{\circ 2001}(z_0,w_0)$ and $\Skew^{\circ 2001+2003}(z_0,w_0)$
 which are colored in red.
 These peculiar values are explained by
 the proof of the Main Theorem (see Proposition \ref{key} below).

\begin{figure}[htbp]
\centering
\def\svgwidth{11cm}
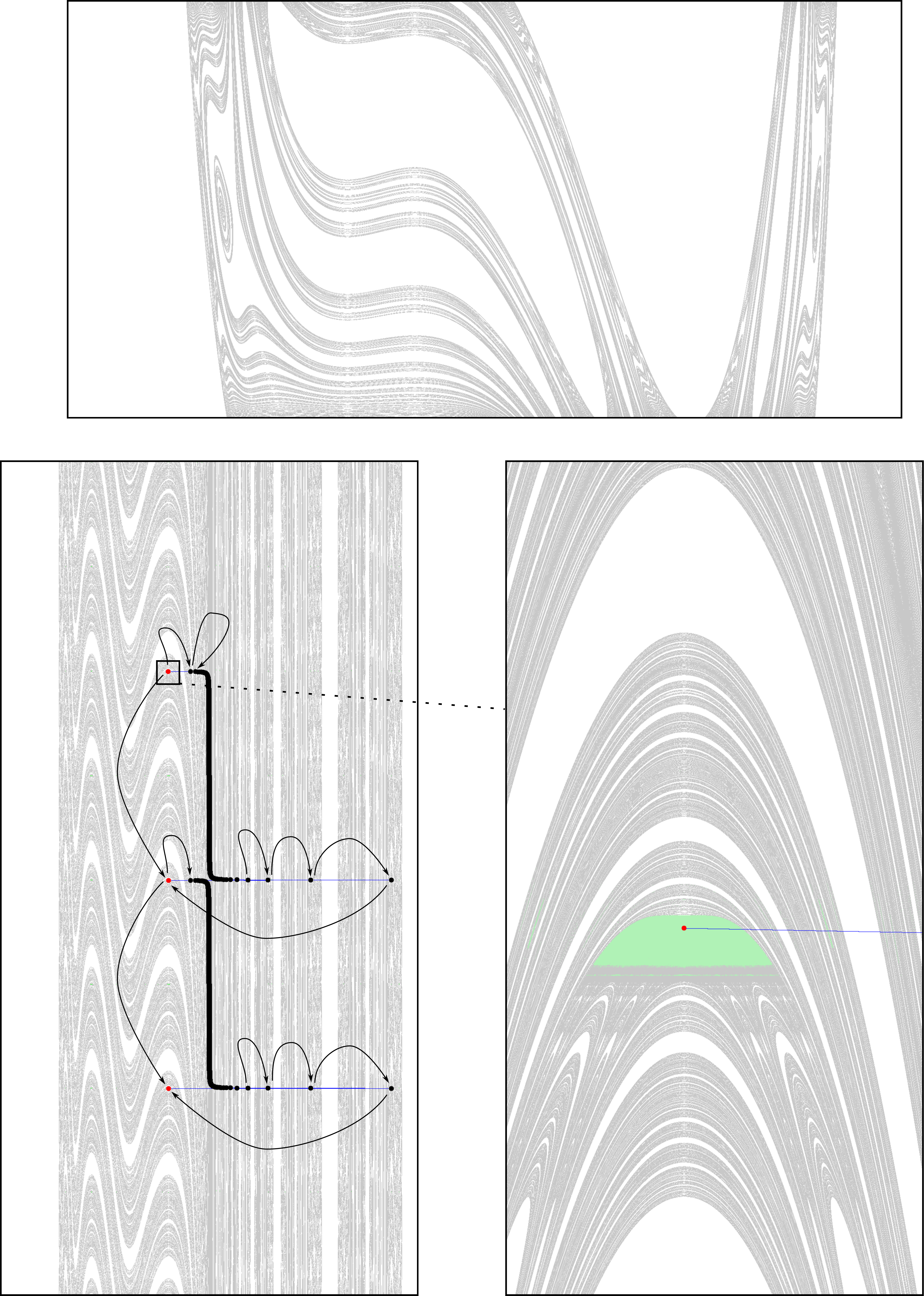
%
\caption{\small A real wandering domain for the map $\Skew$ defined in (\ref{eq:defFreal}).
\label{realwandering}}
\end{figure}

\medskip \begin{center}$\diamond$
\end{center}
\medskip

Using  skew-products to construct new examples is natural as it allows to build on one-dimensional dynamics.
This idea was already used several times in holomorphic dynamics (see e.g. \cite{jonsson, dujardin}) and beyond.

Fatou components of polynomial skew-products were studied in several earlier works. Lilov showed in his thesis
\cite{lilov} that
 skew-products do not have wandering components near a super-attracting invariant fiber.
 In \cite{peters vivas} it was shown that the argument
 used in \cite{lilov} breaks down near an attracting fiber.
 The construction in \cite{peters vivas} uses a repelling fixed point in the invariant fiber
 with multiplier equal to the inverse of the multiplier of the attracting fiber. This resonance between multipliers
  induces a dynamical behavior that cannot occur in one-dimensional dynamics.

  In the same vein, an  invariant fiber at the center of a Siegel disk was used in
  \cite{boc-thaler fornaess peters} to construct a non-recurrent Fatou component with limit set isomorphic to a punctured disk. In that
  construction the invariant fiber also contains a Siegel disk, but with the opposite rotation number.

   We note that the construction presented
  in this paper   uses an invariant parabolic fiber which again contains a parabolic point.

\medskip \begin{center}$\diamond$
\end{center}
\medskip

To explain the notion of Lavaurs map and
the strategy of the proof, we need to recall some facts on parabolic dynamics
(see Appendix \ref{app} for more details). Let $f$ be a polynomial of the form
\begin{equation}
\notag
f(z) = z+z^2+az^3+ \O(z^4)\forsome a\in \C.
\end{equation} and denote by
\[\bpf:=\bigl\{z\in \C~;~f^{\circ n}(z)\overset{\neq}{\underset{n\to +\infty}\longrightarrow} 0\bigr\}\]
the parabolic basin of $0$.
It is known that there is an {\em attracting Fatou coordinate} $\phi_f:\bpf\to \C$ which
conjugates $f$ to the translation $T_1$ by $1$:
\[\phi_f\circ f = T_1\circ \phi_f.\]
This Fatou coordinate may be normalized by\footnote{The branch of $\log$ used in this normalization as well as in
 the next one is the branch defined in $\C\setminus\R^-$ which vanishes at $1$}
\[\phi_f(z) = -\frac{1}{z} -(1-a)  \log \left(-\frac{1}{z} \right)+ \o(1)\as \Re(-1/z)\to +\infty.\]
Likewise, there is a {\em repelling Fatou parameterization}
$\psi_f:\C\to \C$ satisfying
\[\psi_f\circ T_1 = f\circ \psi_f,\]
which may be normalized by
\[-\frac{1}{\psi_f(Z)} = Z +(1-a)  \log (-Z) + \o(1)\as \Re(Z)\to -\infty.\]
The (phase $0$) Lavaurs map $\L_f$ is defined
by\footnote{The reader who is familiar with Lavaurs maps should notice that the
 choice of {phase} was determined by the  normalization of Fatou coordinates.}
\[\L_f:=\psi_f\circ \phi_f:\bpf\to \C.\]
Mappings of this kind   appear when   considering  high iterates of small perturbations of $f$:
this phenomenon is known as {\em parabolic implosion}, and will play a key role in this paper. The reader
is referred  to
the text of  Douady \cite{douady} for a gentle introduction to this topic, and to
\cite{shishikura} for a more detailed presentation by Shishikura.
 (Semi-)parabolic implosion was recently studied  in the context of
 dissipative polynomial automorphisms of $\C^2$ by
Bedford, Smillie and Ueda \cite{bsu} (see also \cite{dl}).

Let us already point out that since our results
 do not fit   into the classical framework,
 our treatment of parabolic implosion will be essentially self-contained.
 As it turns out, our computations bear some similarity with those of \cite{bsu}.

\medskip

%
%
%
%
%
%
%
%

We can now explain the
 strategy of the proof of the Main Theorem. Let $\bpg$ be the parabolic basin of $0$ under iteration of $g$.
If $w\in \bpg$, then  $g^{\circ m}(w)$ converges to~0 like $1/m$.
We want to choose $(z_0,w_0)\in \bpf\times \bpg$ so that the first coordinate of~$\Skew^{\circ m}(z_0, w_0)$
 returns  close to the attracting fixed point of $\L_f$ infinitely many  times. The proof is designed so that the return
 times are the integers $n^2$ for $n\geq n_0$. So, we have to analyze the orbit
 segment between $n^2$ and  $(n+1)^2$, which is of length $2n+1$.

 For large $n$,
 the first  coordinate of $\Skew$ along this orbit segment
is approximately
$$f(z)+ \eps^2 \with
  \frac{\pi}{\eps} = 2n.$$
 The Lavaurs Theorem from parabolic implosion asserts that  if $\frac{\pi}{\eps} = 2n$, then  for large  $n$,
the $(2n)^{\rm th}$ iterate of $f(z)+ \eps^2$ is approximately equal to $\L_f(z)$ on $\bpf$.

Our setting  is slightly different
since $\eps$  keeps decreasing along the orbit. Indeed on the first coordinate  we are taking the composition of $2n+1$
 transformations
$$f(z)+ \eps_k^2\with \frac{\pi}{\eps_k} \simeq 2n+ \frac{k}{n} \and 1 \leq k \leq    2n+1 .$$
The main step of the proof of the Main Theorem consists in a detailed analysis of this perturbed situation. We show that
the decay of $\eps_k$ is   counterbalanced
by taking {\em exactly}  one additional iterate of $\Skew$. The precise statement is the following.


\begin{mproposition}\label{key}
As $n\to +\infty$, the sequence of maps
\[\C^2\ni (z,w) \mapsto \Skew^{\circ 2n+1}\bigl(z,g^{\circ n^2}(w)\bigr)\in \C^2\]
converges locally uniformly in $\bpf\times \bpg$ to the map
\[\bpf\times \bpg\ni (z,w)\mapsto \bigl( \L_f(z),0\bigr)\in \C\times\{0\}.\]
\end{mproposition}

See Figure \ref{fig:key} for a graphical  illustration of  this Proposition.
\begin{figure}[htbp]
\centerline{
\framebox{\includegraphics[height=8cm]{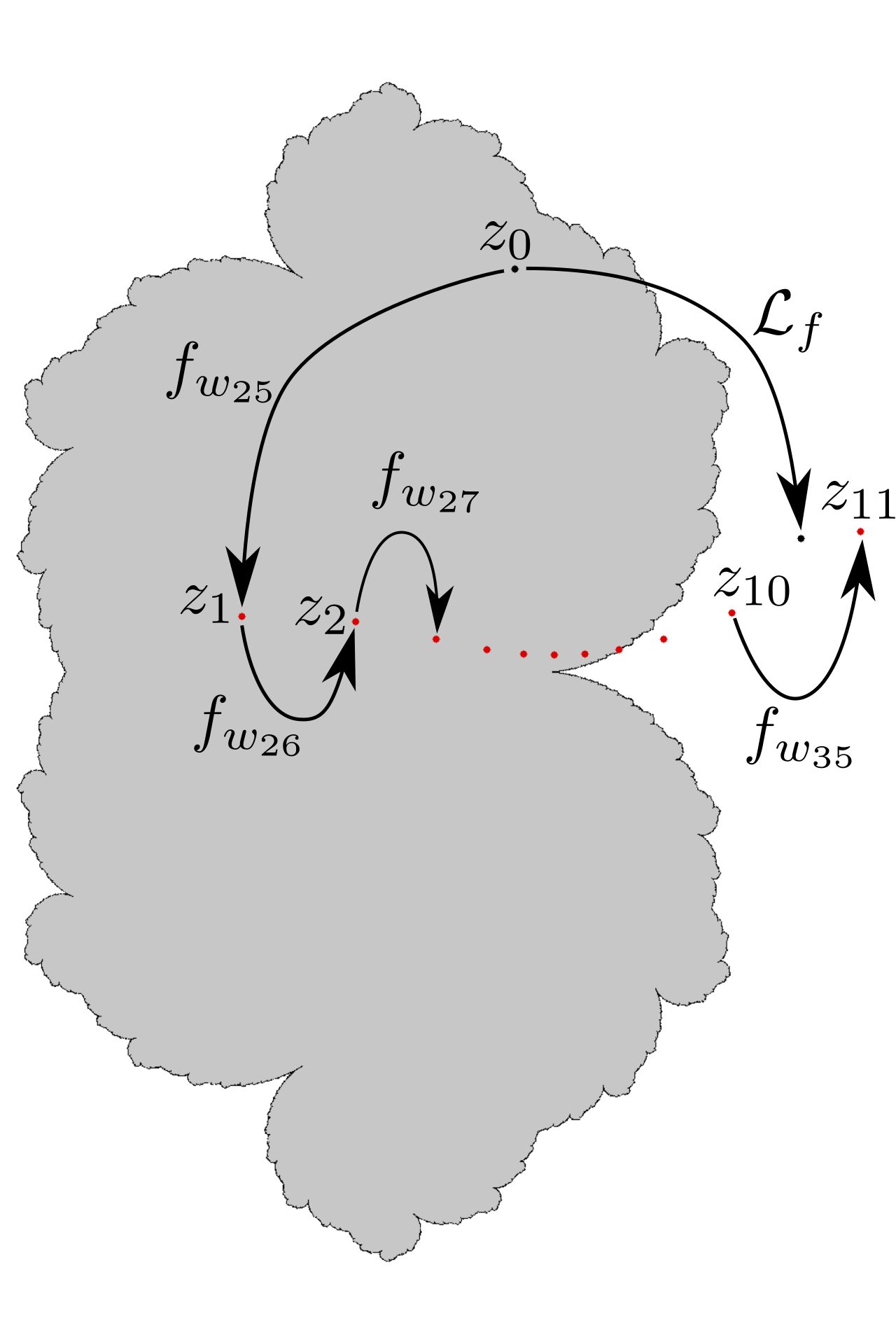}}\quad
\framebox{\includegraphics[height=8cm]{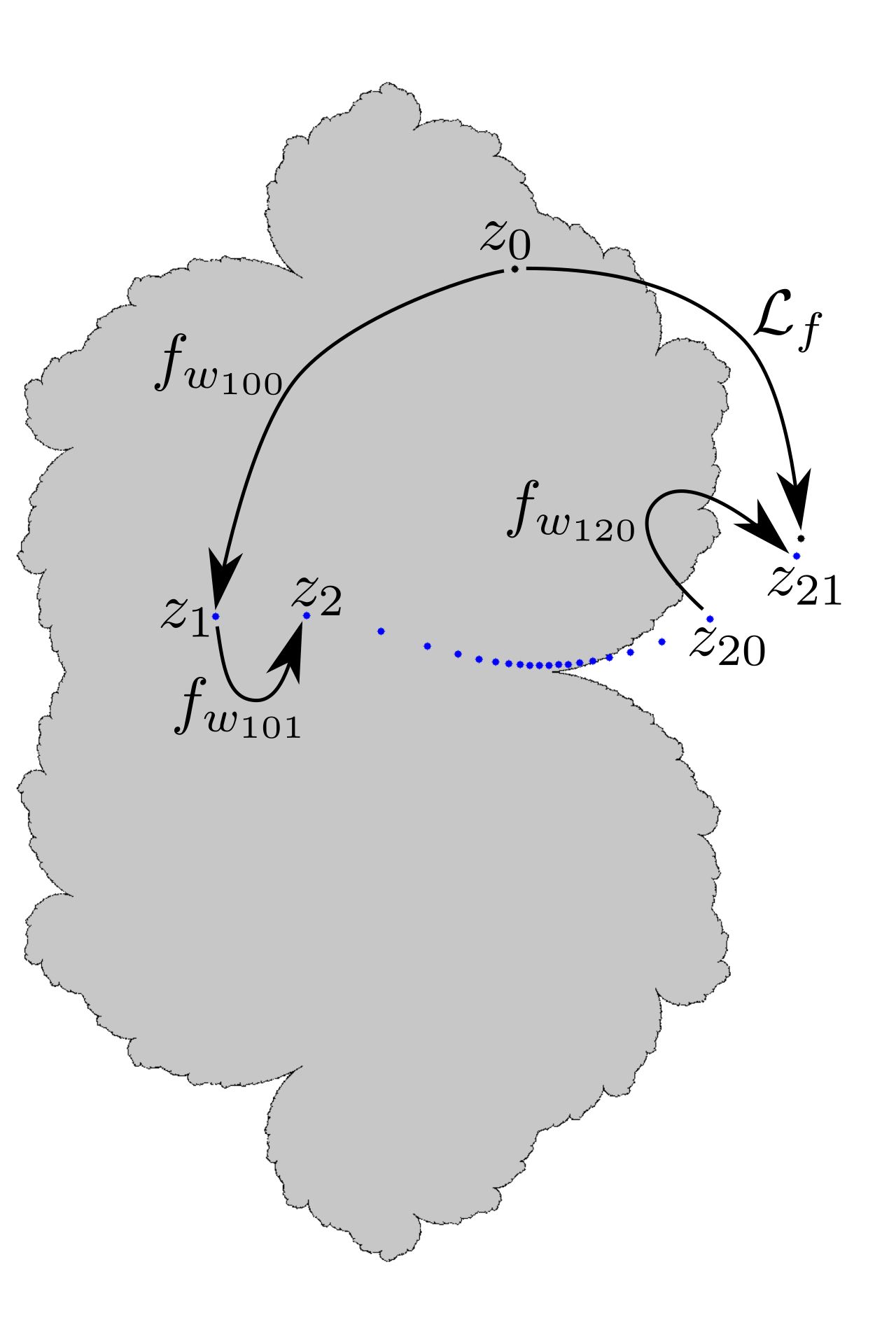}}
}
\caption{\small Illustration of Proposition \ref{key} for $f(z) = z+z^2+0.95z^3$ and $g(w)=w-w^2$. The parabolic basin $\bpf$ is colored in
grey. It is invariant under $f$, but not under $f_w:=f+\frac{\pi^2}{4}w$ for $w\neq 0$. The Lavaurs map $\L_f$ is defined on $\bpf$. The point $z_0=-0.05+0.9\i$ and its image $\L_f(z_0)$ are indicated. The other points are the points $z_k$ which are defined by $\Skew^{\circ k}(z_0,w_{n^2}\big)  = \bigl(z_k,w_{n^2+k}\bigr)$ for $1\leq k\leq 2n+1$ and $w_m=g^{\circ m}(1/2)$. If $n$ is large enough, the point $z_{2n+1}$ is close to $\L_f(z_0)$. Left: $n=5$, there are $11$ red points. Right: $n=10$, there are $21$ blue points. \label{fig:key}}
\end{figure}

\medskip

From this point,   the proof of the Main Theorem
is easily completed: if $\xi$ is an attracting fixed point of $\L_f$ and if $(z_0,w_0)\in \bpf\times \bpg$ is chosen so that $\Skew^{\circ n_0^2}(z_0,w_0)$ is close to $(\xi,0)$ for some large $n_0$, then $\Skew^{\circ (n_0+1)^2}(z_0,w_0)$ gets closer to $(\xi,0)$ and we can repeat the process to get that the sequence $\bigl(\Skew^{\circ n^2}(z_0,w_0)\bigr)_{n\geq 0}$ converges to $(\xi,0)$.  
%
 Since this reasoning is valid on an open set of initial conditions,  these points belong to some  Fatou component.
Simple considerations show that it cannot be preperiodic, and the result follows.

Let us observe that by construction, the $\omega$-limit set of any point in the wandering Fatou component consists of a single 
 two-sided orbit  of $(\xi, 0)$ under $P$, plus the origin. 


\medskip

We give two different approaches for constructing
examples satisfying the assumptions of the Main Theorem. The  next proposition corresponds to Example~\ref{ex:complex}.

\begin{mproposition}\label{attracting}
Consider the cubic polynomial $f:\C\to \C$ defined by
\[f(z):=z+z^2+az^3\with a\in \C.\]
If $r>0$ is sufficiently close to $0$ and $a\in D(1-r,r)$,  then  the Lavaurs map $\L_f:\bpf\to \C$ admits an attracting fixed point.
\end{mproposition}

To prove this proposition we construct a fixed point for the Lavaurs map by perturbation from the degenerate situation $a=1$ and
estimate its multiplier by a residue computation.

\medskip

To construct the real examples of Example \ref{ex:real}, we use the Intermediate Value Theorem to find a real
Lavaurs map with a real periodic critical point.

\begin{mproposition}\label{superattracting}
Consider the degree $4$ polynomial $f:\C\to \C$ defined by
\[f(z):=z+z^2+bz^4\with b\in \R.\]
Then there exists a  parameter $b\in (-8/27,0)$ such that the Lavaurs map
$\L_{f}$ has a fixed critical point in $\R\cap {\cal B}_{f}$.
\end{mproposition}

In particular this fixed point is super-attracting so we are in the situation of the Main Theorem.

\medskip \begin{center}$\diamond$
\end{center}
\medskip

The plan of the paper is the following.
 In Section \ref{sec:thm} we show how to deduce the Main
 Theorem   from  Proposition \ref{key}. In  Section \ref{sec:approxcoord} we develop a formalism of {\em approximate
 Fatou coordinates} in our context, which is necessary for proving Proposition \ref{key},
 which  is itself established in  Section \ref{sec_key}. The proofs in Sections \ref{sec:approxcoord}
  and \ref{sec_key} are rather technical, and are divided into a number of steps. A detailed outline will be given there.
Propositions \ref{attracting} and \ref{superattracting} are respectively established in
 Sections \ref{sec_lavaurs} and \ref{sec:real}.   Appendix \ref{app} summarizes the main properties of Fatou coordinates and Lavaurs maps.

 \section*{Acknowledgements}
This project grew up from  discussions between Misha Lyubich and Han Peters in relation with the work \cite{lyubich peters}.
It evolved to the current list of authors at the occasion of a meeting
of the ANR project LAMBDA in Universit\'e Paris-Est Marne-la-Vall\'ee in April 2014.
We are grateful to Misha Lyubich for his crucial input which ultimately led to the present work.

\section{Existence of wandering domains} \label{sec:thm}

In this section, we prove the Main Theorem assuming Proposition \ref{key}.

\medskip

Let $\xi \in \bpf$ be an attracting fixed point of the Lavaurs map $\L_f$.
 Let $V$ be a disk centered
at $\xi$,  chosen that $\L_f(V)$ is compactly contained in $V$. It follows that
 $\L_f^{\circ k}(V)$ converges to $\xi$ as $k\cv\infty$. Let also  $W \Subset \bpg$ be an arbitrary disk.

Denote by $\pr:\C^2\to \C$   the first coordinate projection, that is
$\pr(z,w):=z.$
According to Proposition \ref{key}, there exists   $n_0 \in \N$ such that for every $n\geq n_0$,
\[ \pr \circ  \Skew^{\circ (2n+1)}(V \times g^{\circ n^2}(W))\Subset V.\]
Let $U$ be a connected component of the open set $\Skew^{-n_0^2}\bigl(V\times g^{\circ n_0^2}(W)\bigr)$.

\begin{lemma}\label{limit}
The sequence $(\Skew^{\circ n^2})_{n\geq 0}$ converges locally uniformly to $(\xi,0)$ on $U$.
\end{lemma}

\begin{proof}
An easy induction shows that for every integer $n\geq n_0$,
\begin{equation}\label{eq:incl}
 \Skew^{\circ n^2} (U)\subseteq V\times
g^{\circ n^2}(W).
\end{equation}
Indeed this holds by assumption for $n=n_0$.  Now if the inclusion is true for some $n\geq n_0$, then
\begin{align*}\pr  \circ \Skew^{\circ (n+1)^2} (U)& = \pr \circ \Skew^{\circ (2n+1)}
\lrpar{ \Skew^{\circ n^2} (U) }\\& \subset \pr \circ \Skew^{\circ (2n+1)} \lrpar{V\times g^{\circ n^2}(W)}\subset V,
\end{align*}
from which \eqref{eq:incl} follows.

From this we get that the sequence $(\Skew^{\circ n^2})_{n\geq 0}$ is uniformly bounded, hence normal, on $U$.   Also,
any cluster value of this sequence of maps is constant and
 of the form $(z,0)$ for some $z\in V$. In addition, $(z,0)$ is a limit value (associated to a subsequence $( {n_k})$) if and only if $\bigl(\L_f(z),0\bigr)$ is a limit value (associated to the subsequence $( {1+n_k})$). We infer that  the set of cluster limits is totally invariant under
 $\L_f:V\to V$, therefore  it must be  reduced to the attracting fixed point $\xi$ of~$\L_f$, and we are done.
\end{proof}

\begin{corollary}
The domain $U$ is contained in the Fatou set of $\Skew$.
\end{corollary}

\begin{proof}
It is well-known in our context that the sequence
  $(\Skew^{\circ m})_{m\geq 0}$ is    locally bounded  on $U$
  if and only if there exists a subsequence  $(m_k)$ such that $(\Skew^{\circ m_k}\rest{U})_{k\geq 0}$  has the same property.
 Indeed since $ \overline {W}$ is compact, there exists $R>0$ such that if $\abs{z}> R$, then for every  $w\in W$, $(z, w)$
 escapes locally uniformly to infinity under iteration. The result follows.
\end{proof}

\begin{proof}[Proof of the Main Theorem]
Let $\Omega$ be the component of the Fatou set $\F_\Skew$ containing $U$. According to Lemma \ref{limit}, for any integer $i\geq 0$,  the
sequence $(\Skew^{\circ (n^2+i)})_{n\geq 0}$ converges locally uniformly to $\Skew^{\circ i}(\xi,0)=\bigl(f^{\circ i}(\xi),0)$ on $U$, hence
 on $\Omega$. Therefore,
 the sequence $(\Skew^{\circ n^2})_{n\geq 0}$ converges locally uniformly to $\bigl(f^{\circ i}(\xi),0)$ on $\Skew^{\circ i}(\Omega)$.

As a consequence,  if $i,j$ are nonnegative integers such that $\Skew^{\circ i}(\Omega)= \Skew^{\circ j}(\Omega)$, then $f^{\circ i}(\xi)= f^{\circ j}(\xi)$, from which we deduce that
 $i=j$. Indeed, $\xi$ belongs to the parabolic basin $\bpf$, and so, it is not (pre)periodic under iteration of $f$. This shows that
$\Omega$ is not (pre)periodic under iteration of $\Skew$: it is a wandering Fatou component for $\Skew$.
%
%
%
%
\end{proof}

\section{Approximate Fatou coordinates\label{sec:approxcoord}}

In this section we study the phenomenon of {\em persistence of Fatou coordinates} in our non-autonomous setting.  As in the Main Theorem,
we consider polynomial mappings $f, g:\C\cv\C$ of the form
$$ f(z)  = z+z^2+az^3+ \O(z^4) \and g(w) = w-w^2+ \O(w^3).$$ We put
$$f_w(z) = f(z) + \frac{\pi^2}{4}w.$$  We want to show that there exist changes of variables $\varphi_w$ and $\varphi_{g(w)}$
which are in a sense approximations to the Fatou coordinates of $f$ in appropriate domains, and such that
$\varphi_{g(w)}\circ f_w\circ \varphi_w^{-1}$ is close to a translation. These change of variables are normalized so that
 $\varphi_{g(w)}\circ f_w\circ \varphi_w^{-1}$ is roughly  defined   in a vertical strip of width~1
 and the translation vector is $\sqrt{w}/2$.   They will be given by explicit formulas: in this respect our approach is similar to that of \cite{bsu}.
 Precise error estimates are required in order to ultimately prove Proposition \ref{key}
 in the next section.

\subsection{Notation}
The following notation will be used throughout this section and the next one (see also Figure \ref{notation}).

First, choose $R>0$ large enough so that $F:Z\mapsto -1/f(-1/Z)$ is univalent on $\C\setminus \overline D(0,R)$,
\[\sup_{|Z|>R} \bigl|F(Z)-Z-1\bigr|<\frac{1}{10}\and \sup_{|Z|>R} \bigl|F'(Z)-1\bigr|<\frac{1}{10}.\]
Denote by $\H_R$ the right half-plane $\H_R:=\bigl\{Z\in \C~;~\Re(Z)>R\bigr\}$ and by $-\H_R$ the left half-plane
$-\H_R:=\bigl\{Z\in \C~;~\Re(Z)<-R\bigr\}$.
Define the  {\em attracting petal} $\Pa$  by
\[\Pa=\left\{z\in \C~;~\Re\left(-\frac{1}{z}\right)>R\right\}\]
Then,
\begin{itemize}
\item the restriction of
the attracting Fatou coordinate $\phi_f$ to the attracting petal $\Pa$
is univalent, and
\item the restriction of the repelling Fatou parameterization $\psi_f$ to the  left half-plane
$-\H_R$ is univalent.
\end{itemize}
We  use the notation $\psi_f^{-1}$ only for the inverse branch $\psi_f^{-1}:\Pr\to -\H_R$  of $\psi_f$ on the
{\rm repelling petal} $ \Pr:=\psi_f(-\H_R)$.
Recall that
 \[\phi_f\circ f = T_1\circ \phi_f, \quad f\circ \psi_f = \psi_f\circ T_1,\]
 \[ \phi_f(z)\sim -\frac{1}{z}\as\Re\left(-\frac{1}{z}\right)\to +\infty
 \and \psi_f(Z)\sim   -\frac{1}{Z}\as\Re(Z)\to -\infty.\]

Next, for $r>0$ we set $ \Pg:=D(r,r)$ and fix
  $r>0$ small enough   that
\[\Pg\subset {\cal B}_g\and g(\Pg)\subset \Pg.\]
For the remainder of the article, we assume that $w\in \Pg$, whence $g^{\circ m}(w)\to 0$ as $m\to +\infty$. The notation $\sqrt w$ stands for the branch of the square root on $B_r$ that has   positive real part.

Finally, we fix a real number
\begin{equation} \label{eq:alpha}
\frac12< \alpha <\frac23.
\end{equation}
The relevance of   this range of values for $\alpha$  will be made clearer during the proof.
For $w\in \Pg$,  we let
\[r_w:=|w|^{(1-\alpha)/2}\underset{w\to 0}\longrightarrow 0 \and R_w:=|w|^{-\alpha/2}\underset{w\to 0}\longrightarrow +\infty.\]
Define  $\mR_w$  to be the rectangle
\begin{equation}\label{eq:Rw}
\mR_w:=\left\{\mz\in \C~;~\frac{r_w}{10}<\Re(\mz)<1-\frac{r_w}{10}\text{ and } -\frac{1}{2}<\Im(\mz)<
\frac{1}{2}\right\},
\end{equation}
 and let $D_w^\att$ and $D_w^\rep$ be the disks
\[D_w^\att:=D\left(R_w,\frac{R_w}{10}\right)\and
D_w^\rep:=D\left(-R_w,\frac{R_w}{10}\right).\]

In this section, the notation   $\O(h)$  means a quantity that is defined for $w\in \Pg$ close enough to zero and is
bounded by $C\cdot h$ for a constant $C$ which does not depend on $w$.
As usual, $\o(h)$ means a quantity such that $\o(h)/h$ converges to zero as $w\cv0$.

\subsection{Properties of approximate Fatou coordinates}\label{subs:properties}

Our purpose in this paragraph is to state precisely the properties of the approximate Fatou coordinates, in an axiomatic fashion.
The actual definitions as well as the proofs will be detailed afterwards.

 The  claim is that there exists  a family of domains $(V_w)$ and
univalent maps $(\varphi_w:V_w\to \C)$ parameterized by $w\in \Pg$, satisfying the following three properties.

\begin{property}[Comparison with the attracting Fatou coordinate]\label{compatt}
As $w\to 0$ in $\Pg$,
\[D_w^\att\subset \phi_f\left(V_w\cap \Pa\right)\and \sup_{Z\in D_w^\att}\left|\frac{2}{\sqrt w} \cdot  \varphi_w\circ \phi_f^{-1}(Z) -Z\right| \longrightarrow 0.\]
\end{property}

This means
that $\frac{2}{\sqrt{w}}   \varphi_w $ approximates the Fatou coordinate on the attracting side.
A similar result holds on the repelling side.

\begin{property}[Comparison with the repelling Fatou coordinate]\label{comprep}
As $w\to 0$ in $\Pg$, we have that
\[1+\frac{\sqrt w}{2}\cdot D_w^\rep\subset \varphi_w\left(V_w\cap \Pr\right)\]
and
\[\sup_{Z\in D_w^\rep}\left|\psi_f^{-1}\circ \varphi_w^{-1}\left(1+\frac{\sqrt w}{2}   Z\right)-Z\right| \longrightarrow 0.\]
\end{property}

Finally, the last property asserts that $\frac{2}{\sqrt{w}} \varphi_w $ is almost a Fatou coordinate.

\begin{property}[Approximate conjugacy to a translation]\label{approx}
As $w\to 0$ in $\Pg$, we have that
\[\mR_w\subset \varphi_w(V_w), \quad f_{w}\circ \varphi_w^{-1}(\mR_w)\subset V_{g(w)}\]
and
\[\sup_{\mz\in \mR_w} \left|\varphi_{g(w)} \circ f_{w}\circ \varphi_w^{-1}(\mz) - \mz - \frac{\sqrt w}{2} \right|= \o(w).\]
\end{property}

\medskip

To improve the readability of the proof, which involves several changes of coordinates,
we adopt the following typographical  convention:
\begin{itemize}
\item
block letters (like $z$, $V_w$, \ldots)
are used for objects which are thought of as living
in the initial coordinates;
\item script like letters (like $\mz$, $\mR_w$, \ldots) are used for objects  living  in approximate Fatou coordinates;
\item  the coordinate $Z$ is used for  the actual Fatou coordinate.
\end{itemize}
This gives rise to expressions like $\phi_f(z)  = Z$ or $\varphi_w(z) = \mz$.

\begin{figure}[htbp]
\centering
\def\svgwidth{14cm}
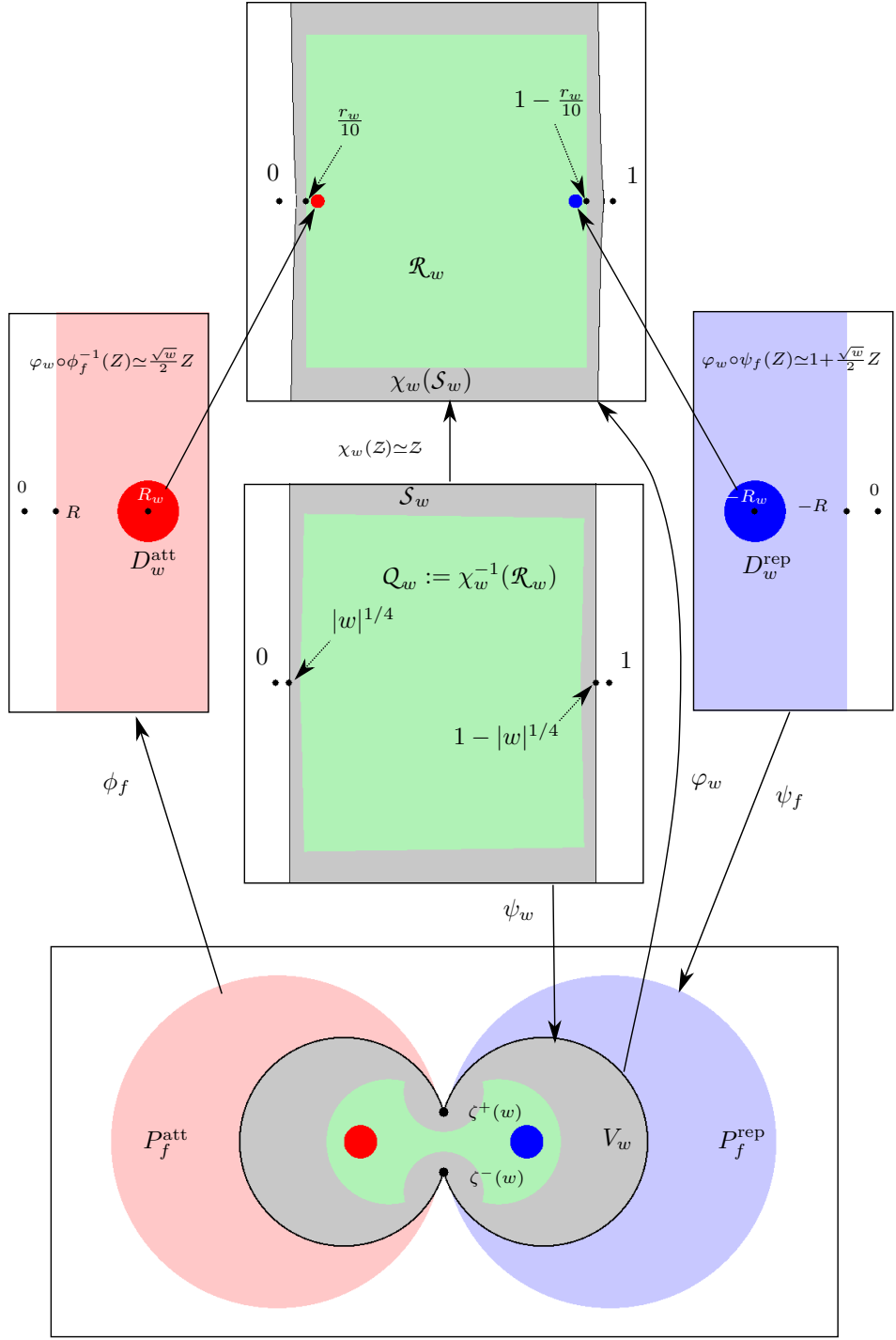
\caption{{\small Changes of coordinates used in the proof. \label{notation}}}
\end{figure}

\subsection{Definition of the approximate Fatou coordinates}

The skew-product $\Skew$ fixes the origin and  leaves the line $\{w=0\}$   invariant.
We may wonder whether there are other {\it parabolic invariant curves} near the origin, in the
sense of   \'Ecalle \cite{Ec} and   Hakim \cite{h}.

\begin{question}
Does there exist  holomorphic maps
 $\xi^\pm : \Pg\to \C$ such that $\xi^\pm(w)\to 0$ as $w\to 0$ and such that  $f_w\circ \xi^\pm(w)= \xi^\pm\circ g(w)$ for $w\in \Pg$?
\end{question}

 We shall content ourselves with the following weaker result.

\begin{lemma}\label{lem:zetapm}
Let $\zeta^\pm:\Pg\to \C$ be defined by
\[\zeta^\pm(w) = \pm c_1\sqrt{w} + c_2 w\with c_1 = \frac{\pi \i}{2}\and c_2 =  \frac{a\pi^2}{8}-\frac{1}{4}.\]
Then,
\[f_{w}\circ \zeta^\pm (w) = \zeta^\pm \circ g(w) + \O(w^2).\]
\end{lemma}

\begin{proof}
An elementary computation shows that
$$f_w\circ \zeta^\pm(w) = \ds \pm c_1\sqrt w + \left(c_2+c_1^2+\frac{\pi^2}{4}\right) w \pm (ac_1^3 + 2c_1c_2) w\sqrt w +\O(w^2).$$
On the other hand,
\[ \sqrt{g(w)} = \ds \sqrt w-\frac{1}{2}w\sqrt w + \O(w^2), \]
so
\[\zeta\circ g(w) = \ds \pm c_1\sqrt w + c_2 w \mp \frac{c_1}{2}w\sqrt w+ \O(w^2).\]
Thus the result follows from our choice of $c_1$ and $c_2$  since
\[c_2+c_1^2+\frac{\pi^2}{4} = c_2\and ac_1^3 + 2c_1c_2= -\frac{c_1}{2}.\qedhere\]
\end{proof}

Let $\psi_w:\C\to \P^1(\C)\setminus\{\zeta^+(w),\zeta^-(w)\}$ be the universal cover defined by
\begin{equation}
\label{eq:psiw}
\psi_w(\mz) := \frac{\zeta^-(w)\cdot \e^{2\pi \i \mz} - \zeta^+(w)}{\e^{2\pi \i \mz}-1} = \i c_1\sqrt w \cot(\pi \mz) + c_2w.
\end{equation}
This universal cover restricts to a univalent map on the vertical strip
\[\mS_0:=\bigl\{\mz\in \C~;~0<\Re(\mz)<1 \bigr\} , \] with
inverse given by
\[\psi_w^{-1}(z)= \frac{1}{2\pi \i} \log \left(\frac{z-\zeta^+(w)}{z-\zeta^-(w)}\right).\]
In this formula, $\log(\cdot)$ stands for  the branch of the logarithm   defined on $\C\setminus \R^+$ and such that
$\log(-1) = \pi\i$.

For $w\in \Pg$, let $\chi_w:\mS_0\to \C$ be the map defined by
\begin{equation}
\label{eq:defchiw}
\chi_w(\mz):= \mz- \frac{\sqrt{w}  (1-a)}{2}\log \left(\frac{2\sin(\pi \mz)}{\pi \sqrt w}\right),
\end{equation}
where in this formula
the branch of $\log$ is defined on $ \frac{1}{\sqrt w} (\C\setminus\R^-)$ and vanishes at~$1$.

Now introduce
\[\mS_w:=\bigl\{\mz\in \C~;~|w|^{1/4}<\Re(\mz)<1-|w|^{1/4}\bigr\}\]
and its  image under $\psi_w$:
\[V_w:=\psi_w(\mS_w)\subset \C.\]

\begin{lemma}\label{lem:chiw}
If $w\in \Pg$ is close enough to $0$, then $\chi_w:\mS_w\to \C$ is univalent. In addition, $\chi_w$ is close to the identity in the following sense:
$$\text{if }\mz\in \mS_w \cap \set{\mz, \ \Im(\mz)<1} \text{ then }
 \chi_w (\mz) = \mz + \O \lrpar{\abs{w}^{1/2}\log \abs{w}} = \mz +  \o(r_w).$$
\end{lemma}

\begin{proof}
Observe that
\[\chi_w'(\mz) = 1- \frac{\sqrt w  (1-a)\pi }{2}\cot(\pi \mz)\and \sup_{\mz\in \mS_w} \bigl|\cot (\pi \mz)\bigr| \in \O\left(\abs{w}^{-1/4}\right).\]
As a consequence,
\[\sup_{\mz\in \mS_w}\bigl|\chi_w'(\mz)-1\bigr|\in \O\lrpar{\abs{w}^{1/4}}.\]
If $\chi_w(\mz_1)=\chi_w(\mz_2)$, then
\[|\mz_2-\mz_1| = \bigl|(\chi_w(\mz_1)-\mz_1)-(\chi_w(\mz_2)-\mz_2)\bigr| \leq \sup_{[\mz_1,\mz_2]} \bigl|\chi_w'(\mz)-1\bigr| \cdot |\mz_2-\mz_1| .\]
When $w$ is sufficiently close to $0$, the supremum is smaller than $1$ and we necessarily have $\mz_1=\mz_2$.

The second assertion of the lemma follows directly from the definition of $\chi_w$ and the fact that on $\mS_w \cap \set{ \Im(\mz)<1} $,
$\abs{\sin(\pi \mz)}\geq c \abs{w}^{1/4}$ for some positive constant~$c$.
\end{proof}

From now on, we assume that $w$ is sufficiently close to $0$ so that $\chi_w:\mS_w\to \C$ is univalent.

\begin{definition}
The approximate Fatou coordinates $\varphi_w$ are the maps
\[\varphi_w := \chi_w\circ \psi_w^{-1} : V_w\to \C \with w\in B_r.\]
\end{definition}

We need to prove that these approximate Fatou coordinates satisfy Properties \ref{compatt},  \ref{comprep} and \ref{approx}.

\subsection{Comparison with the attracting Fatou coordinate}\label{subs:comparingatt}

In this paragraph, we prove that  the approximate Fatou coordinate $\varphi_{w}$ satisfies Property \ref{compatt}, namely
when $w\to0$ in $\Pg$,    $D_w^\att\subset \phi_f(V_w)$
and
\begin{equation}\label{eq:supdatt}
\sup_{Z\in D_w^\att}\left|\frac{2}{\sqrt w}\cdot \varphi_w\circ \phi_f^{-1}(Z) -Z\right| \longrightarrow 0.
\end{equation}
Recall that $R_w=|w|^{-\alpha/2}$,  $r_w=|w|^{1/2} R_w=|w|^{(1-\alpha)/2}$, and $D_w^\att = D(R_w, R_w/10)$.

\medskip
\begin{proof}[Proof of Property \ref{compatt}]~\\
\noindent{\bf Step 1.} Let us first prove that $D_w^\att\subset \phi_f(V_w)$.
Note that $R_w\to +\infty$ as $w\to 0$, hence $D_w^\att\subset \phi_f(\Pa)$ for $w$ close  to $0$.
If $z\in \phi_f^{-1}(D_w^\att)$, then
\[\phi_f(z) = - \frac{1}{z}+\o\left(\frac{1}{z}\right)=\O\left(R_w\right).\]
In addition,
\[ \zeta^\pm(w) = \pm \frac{\pi\i}{2}\sqrt{w}  \bigl(1+\o(1)\bigr)\and \frac{\zeta^\pm(w)}{z}=\O\left(r_w\right).\]
In particular,
\begin{align*}
\log\left(\frac{z-\zeta^+(w)}{z-\zeta^-(w)}\right)
&=\log\left(1-\frac{\zeta^+(w)}{z}\right)-\log\left(1-\frac{\zeta^-(w)}{z}\right)\\
&=-\frac{\zeta^+(w)}{z} +\frac{\zeta^-(w)}{z}+ \O\left(r_w^2\right) =-\pi\i \frac{\sqrt{w}}{z}+\O\left(r_w^2\right).
\end{align*}
Since $\alpha>1/2$, we have that $r_w^2 = |w|^{1-\alpha} = \o\left(|w|^{1/2}\right)$, and it follows that
\begin{equation}\label{eqZ}
\mz:=\frac{1}{2\pi \i}\log \left(\frac{z-\zeta^+(w)}{z-\zeta^-(w)}\right) =- \frac{\sqrt{w}}{2z}+\O\left(r_w^2\right) = - \frac{\sqrt{w}}{2z}+ \o\left(|w|^{1/2}\right).
\end{equation}
So, the real part of $\mz$ is comparable to $r_w$ and since $|w|^{1/4} = \o\left(r_w\right)$, we get that
 $\mz\in \mS_w$, whence $z=\psi_w(\mz)\in V_w$ for $w\in \Pg$ close enough to $0$.

\medskip\noindent{\bf Step 2.} We now establish \eqref{eq:supdatt}. Note that
\[\sup_{Z\in D_w^\att}\left|\frac{2}{\sqrt w} \cdot \varphi_w\circ \phi_f^{-1}(Z) -Z\right|  =
\sup_{z\in \phi_f^{-1}(D_w^\att)}\left|\frac{2}{\sqrt w} \cdot \varphi_w(z)- \phi_f(z)\right|.\]
Observe first that  when $w$ tends to $0$, the domain $\phi_f^{-1}(D_w^\att)$ also tends to $0$. So, if $z\in \phi_f^{-1}(D_w^\att)$, then
\[\phi_f(z) =- \frac{1}{z}-(1-a)\log\left(-\frac{1}{z}\right)+\o(1).\]
On the other hand,
\[\frac{2\sin(\pi \mz)}{\pi\sqrt{w}} =\frac{2}{\sqrt w} \bigl(\mz+\o(\mz)\bigr) =  -\frac{1}{z}+\o\left(\frac{1}{z} \right).\]
Thus, together with the estimate  \eqref{eqZ} we infer that
\begin{equation}\label{phiw}
\begin{aligned}
\varphi_{w}(z) =  \chi_{w}(\mz)
&= \mz-\frac{\sqrt{w}\cdot (1-a)}{2}\log\left(\frac{2\sin(\pi \mz)}{\pi\sqrt{w}}\right)\\
& = - \frac{\sqrt{w}}{2z}+\o\left(|w|^{1/2}\right)-\frac{\sqrt{w}\cdot (1-a)}{2}\log\left(-\frac{1}{z}+\o\left(\frac{1}{z} \right)\right)\\
&= \frac{\sqrt{w}}{2}  \left(-\frac{1}{z}-(1-a) \log\left(-\frac{1}{z}\right)+\o\left({1}\right)\right)\\
&=  \frac{\sqrt{w}}{2}  \bigl(\phi_f(z)+\o(1)\bigr),
\end{aligned}
\end{equation}
which  completes the proof.
\end{proof}

\subsection{Comparison with the repelling Fatou coordinate}

In this paragraph, we deal with    Property \ref{comprep}, that is, we wish to prove   that
  as $w\to 0$ in $\Pg$,
\[\mD'_w:=1+\frac{\sqrt w}{2}\cdot D_w^\rep\subset \varphi_w\left(V_w\cap \Pr\right)\]
and
\[\sup_{Z\in D_w^\rep}\left|\psi_f^{-1}\circ \varphi_w^{-1}\left(1+\frac{\sqrt w}{2}   Z\right)-Z\right| \longrightarrow 0.\]
The proof is rather similar to that of Property \ref{compatt}.

\medskip
\begin{proof}[Proof of Property \ref{comprep}]~\\
\noindent{\bf Step 1.}
Let us first prove that for $w\in\Pg$ close enough to $0$, the disk $\mD'_w$ is contained in  $\varphi_w(V_w)$. Note that
 with  $r_w = |w|^{1/2}  R_w = |w|^{(1-\alpha)/2}$ as before, we have
\[\mD'_w= D\left(1-\frac{\sqrt w  R_w}{2},\frac{r_w}{20}\right).\]
Since $\alpha>1/2$, we have that $|w|^{1/4}=\o(r_w)$. Furthermore,
   $\Re(\sqrt w)>\frac{\sqrt 2}{2}|w|^{1/2}$  for $w\in \Pg$, hence
\[  \mD''_w:=\ds D\left(1-\frac{\sqrt w  R_w}{2},\frac{r_w}{10}\right)\subset \mS_w.\]
In addition, by Lemma \ref{lem:chiw},  $\chi_w(\mz) = \mz+\o(r_w)$ for $\mz\in \mD''_w$, so
\[\mD'_w\subset \chi_w(\mD''_w)\subset \chi_w(\mS_w) = \varphi_w(V_w).\]

\medskip\noindent{\bf Step 2.} Given $Z\in D_w^\rep$, we  set
$$
\mx:= \chi_w^{-1}\left(1+\frac{\sqrt w}{2} Z\right).  $$
Note that   $\sqrt{w}Z$ has  modulus equal to  $r_w$.
Also, put
\begin{equation}\label{eq:mx} z:=\varphi_w^{-1}\left(1+\frac{\sqrt w}{2}  Z\right) =
\i c_1\sqrt w \cot(\pi \mx) + c_2w.\
\end{equation}
By Lemma \ref{lem:chiw} we have that
\[\mx -1 =  \frac{\sqrt w}{2} Z \cdot \bigl(1+\o(1)\bigr)=  \O(r_w),\]
hence
\[
\cot(\pi \mx) = \cot\bigl(\pi (\mx-1)\bigr)=\frac{2}{\pi \sqrt w Z}  \bigl(1+\o(1)\bigr).\]
 Remembering   that $c_1= {\pi\i}/2$, from   \eqref{eq:mx}
we deduce that
\[z =  \frac{2\i c_1}{\pi Z}  \cdot \bigl(1+\o(1)\bigr)
\with \frac{2\i c_1}{\pi Z}
= -\frac{1}{Z}\in D\left(|w|^{\alpha/2},\frac{|w|^{\alpha/2}}{2}\right).\]
So  when  $w\in \Pg$ is close enough to $0$, we find  that $z\in \Pr$ and
\[\psi_f^{-1}(z)  = -\frac{1}{z}-(1-a) \log\left(\frac{1}{z}\right)+\o(1).\]
Moreover, 
\[\frac{2\sin(\pi \mx)}{\pi\sqrt{w}} = -\frac{2\sin(\pi (\mx-1))}{\pi\sqrt{w}} = -\frac{2}{\sqrt{w}}\left(\frac{\sqrt w}{2} Z \cdot
\big(1+\o(1)\big)
\right) = \frac{1}{z}+\o\lrpar{\unsur{z}}.\]
Finally, as in (\ref{phiw}), we compute
\begin{align*}
Z = \frac{2}{\sqrt w}\bigl( \chi_{w}(\mx) -1\bigr)
&= -\frac{1}{z} + \o(1) -(1-a) \log\left(\frac{1}{z}  + \o\lrpar{\unsur{z}} \right)\\
&=  \psi_f^{-1}(z) +\o(1) \\
&= \psi_f^{-1}\circ \varphi_w^{-1}\left(1+\frac{\sqrt w}{2}   Z\right)+\o(1).
\end{align*}
This completes the proof of Property \ref{comprep}.
\end{proof}

\subsection{Approximate translation property}

In this paragraph, we prove that the approximate Fatou coordinate  $\varphi_w$ satisfies Property \ref{approx}, that is:
as $w\to 0$ in $\Pg$, the inclusions
$ \mR_w\subset \varphi_w(V_w)$ and $f_{w}\circ \varphi_w^{-1}(\mR_w)\subset V_{g(w)}$
hold (recall that the rectangle $\mR_w$ was defined in \eqref{eq:Rw}), and
\[\sup_{\mz\in \mR_w} \left|\varphi_{g(w)} \circ f_{w}\circ \varphi_w^{-1}(\mz) - \mz - \frac{\sqrt w}{2} \right|= \o(w).\]

\begin{proof}[Outline of the proof]
Let
\[\psi^0:=\psi_w, \ \psi^2:= \psi_{g(w)},\ \chi^0:=\chi_w,\text{and }  \chi^2:=\chi_{g(w)}.\]
To handle the fact that $f_w\circ\zeta^\pm$ is not exactly equal to $\zeta^\pm\circ g$ (see Lemma \ref{lem:zetapm}),
rather than dealing directly with  $\psi_2\circ f_w \circ \psi_0^{-1}$,
we introduce an intermediate change of coordinates
\[\psi^1:\C\ni \mz\mapsto \frac{f_w\bigl(\zeta^-(w)\bigr)\cdot \e^{2\pi \i \mz} -
f_w\bigl(\zeta^+(w)\bigr)}{\e^{2\pi \i \mz}-1}\in \P^1(\C)\setminus\bigl\{f_w\bigl( \zeta^+(w)\bigr),f_w\bigl(\zeta^-(w)\bigr)\bigr\} . \]
Let $\mH$ be the horizontal strip
\[\mH:=\bigl\{\mz\in \C~;~-1<\Im(\mz)<1\bigr\}.\]
We will see that there are lifts $\mf^0 : \mS_w \cv \C$, $\mf^1: \mH\cv\C$ and a map $\mf: \mR_w\cv\C$ such that the following diagram commutes:
\[\diagram
 \mR_w \ar[rrrr]^\mf &&&& \C \\
\mQ_w  \rrto^{\mf^0}\uto_{\chi^0}\dto^{\psi^0} &&  \mf^0(\mQ_w)  \dto^{\psi^1}\rrto^{\mf^1} && \mS_{g(w)} \dto_{\psi^2}\uto^{\chi^2}\\
\psi^0(\mQ_w)  \rrto_{f_w} \ar@/^2.5pc/[uu]^{\varphi_w}&& V_{g(w)} \rrto_{{\rm id}} && V_{g(w)} \ar@/_2.5pc/[uu]_{\varphi_{g(w)}}.
\enddiagram
\]

\medskip\noindent{\bf Step 1.}
We prove that $\mR_w\subset  \varphi_w(V_w)=\chi^0(\mS_w)$ and that
\begin{equation}
\notag \sup_{\mz\in \mR_w}\left|\varphi_w^{-1}(\mz)\right|=\O\left(|w|^{\alpha/2}\right).
\end{equation}
Define $\mQ_w:=(\chi^0)^{-1}(\mR_w)\subset \mS_w$.

\medskip\noindent{\bf Step 2.}
We define $\mf^0$ on $\mS_w$ and prove that for $\mz\in \mQ_w$,
\begin{equation}\label{eq:F0}
\mf^0(\mz) = \mz+\frac{\sqrt{w}}{2} + \frac{\pi (1-a)w}{4}\cot(\pi \mz) + \o(w).
\end{equation}
In particular, for $w$ close enough to $0$, $\mf^0(\mz)= \mz+\O\left(|w|^{1/2}\right)$, hence
$\mf^0(\mQ_w)\subset \mH.$

\medskip\noindent{\bf Step 3.}
We define $\mf^1$ on  $\mH$ and prove that for $\mz\in \mH$,
\[\mf^1(\mz) = \mz+\o(w).\]
In particular, for $w$ close enough to $0$, $\mf^1\circ \mf^0(\mz) =  \mf^0(\mz) + \o(w)$, from which we
deduce
 that
\[\mf^1\circ \mf^0(\mQ_w)\subset \mS_{g(w)}\whence f_{w}\circ \varphi_w^{-1}(\mR_w)\subset V_{g(w)}.\]

\medskip\noindent{\bf Step 4.}  We use $\chi^0$ and $\chi^2$ to eliminate the third term on the right
hand side of \eqref{eq:F0}. Specifically,
we define $\mf$  on $\mR_w$ and prove that for $\mz\in \mR_w$,
\[\mf(\mz) = \mz+\frac{\sqrt w}{2}+\o(w).\]
Thus Property \ref{approx} is established.
\end{proof}

\subsubsection{Proof of Step 1}
We prove that $\mR_w\subset \varphi_w(V_w)$ and that
\[\sup_{\mz\in \mR_w}\left|\varphi_w^{-1}(\mz)\right|=\O\left(|w|^{\alpha/2}\right). \]

Let $\mR'_w\subset \mS_w$ be the rectangle
\[\mR'_w:=\left\{\mz\in \mS_w~;~\frac{r_w}{20}<\Re(\mz)<1-\frac{r_w}{20}\text{ and }-1<\Im(\mz)<1\right\},\]
with, as before,  $r_w:=|w|^{(1-\alpha)/2}$.
We see that  $\mR_w\subset \mR'_w$ and the distance between the boundaries is $r_w/20$.
On the other hand, by Lemma \ref{lem:chiw},  for $\mz\in \mR'_w$,
  $\chi_w(\mz) =\mz+\o(r_w)$.
From this it follows that $\chi_w(\partial \mR'_w)$ surrounds $\mR_w$, whence
\[\mR_w\subset \chi_w(\mR'_w)\subset \chi_w(\mS_w)=\varphi_w(V_w), \] as desired.

To prove the estimate on $\varphi_w^{-1}(\mz)$, define $\mQ_w:=\chi_w^{-1}(\mR_w)$.
Since $\varphi_w = \chi_w\circ \psi_w^{-1}$,  we see that
$\varphi_w^{-1}(\mR_w)= \psi_w(\mQ_w)$.
The above sequence of inclusions shows
that   $\mQ_w\subset \mR'_w$.  Thus from
\[\psi_w(\mz) = \i c_1\sqrt w \cot(\pi \mz) +c_2 w,\]
we infer that
\begin{equation} \label{eq:step1}
\begin{aligned}
 \sup_{\mz\in \mR_w}\left|\varphi_w^{-1}(\mz)\right|= \sup_{\mz\in \mQ_w}\bigl|\psi_w(\mz)\bigr| &\leq
 \sup_{\mz\in \mR'_w}\bigl|\psi_w(\mz)\bigr|  \\
& \notag = \O\left(\frac{|w|^{1/2}}{r_w}\right) +\O(w)=\O\left(|w|^{\alpha/2}\right).
\end{aligned}\end{equation}
This completes the proof of Step 1.

\subsubsection{Proof of Step 2}
We define $\mf^0$ on $\mS_w$ and prove that for $\mz\in \mQ_w$,
\[
\mf^0(\mz) = \mz+\frac{\sqrt{w}}{2} + \frac{\pi (1-a)w}{4}\cot(\pi \mz) + \o(w) = \mz+\O\left(|w|^{1/2}\right).
\]

\medskip\noindent{\bf Step 2.1.} We first define $\mf^0$.
It will be convenient to set $w=\eps^2$ so that expansions with respect to $\sqrt w$ become expansions with respect to $\eps$. Set
\[\zeta_0^\pm (\eps):= \zeta^\pm(\eps^2) = \pm c_1\eps + c_2\eps^2, \
 \zeta_1^\pm(\eps):=f_{\eps^2}\circ \zeta_0^\pm(\eps)\text{ and } \zeta_2^\pm(\eps):=\zeta^\pm\circ g(\eps^2).\]
Choose $r_1>0$ small enough  so that  the only preimage of $0$ by $f$ within $D(0,2r_1)$ is~$0$.
Choose $r_2>0$ so that for $\eps\in D(0,r_2)$, the only preimage of $\zeta_1^\pm(\eps)=f_{\eps^2} \big( \zeta_0^\pm(\eps)\big)$ under  $f_{\eps^2}$ within $D(0,r_1)$ is $\zeta_0^\pm(\eps)$.
The function
\[(z,\eps)\longmapsto \frac{f_{\eps^2}(z) - \zeta_1^+(\eps)}{z-\zeta_0^+(\eps)}\cdot \frac{z-\zeta_0^-(\eps)}{f_{\eps^2}(z) - \zeta_1^-(\eps)}\]
extends holomorphically to $\Delta:=D(0,r_1)\times D(0,r_2)$ and does not vanish there. In addition, it identically takes the value $1$ for $\eps=0$. We set
\[\mmu:(z, \eps) \longmapsto \frac{1}{2\pi \i}\log \left( \frac{f_{\eps^2}(z) - \zeta_1^+(\eps)}{z-\zeta_0^+(\eps)}\cdot \frac{z-\zeta_0^-(\eps)}{f_{\eps^2}(z) - \zeta_1^-(\eps)}\right)\]
where the branch of logarithm is chosen so that $\mmu(z,0)\equiv 0$.
Consider the map $\mf^0$ defined on $\mS_w$ by
\[\mf^0(\mz):= \mz+\mmu\bigl(\psi^0(\mz),\eps\bigr).\]
Then, for $\mz \in \mS_w$, set
\[z:=\psi^0(\mz) \in V_w\sothat  \mz=  \frac{1}{2\pi \i}\log \left(\frac{z-\zeta_0^+(\eps)}{z-\zeta_0^-(\eps)} \right).\]
As $\mz$ ranges in $\mS_w$,  $z$ avoids the points $\zeta_0^\pm(\eps)$ and remains in a small disk around~$0$, thus
$f_{\eps^2}(z)$ avoids the points $\zeta_1^\pm(\eps) = f_{\eps^2}\bigl(\zeta_0^\pm(\eps)\bigr)$. So we can define
\[\mz_1:= \frac{1}{2\pi \i}\log \left(\frac{f_{\eps^2}(z)-\zeta_1^+(\eps)}{f_{\eps^2}(z)-\zeta_1^-(\eps)} \right)\]
where the branch is chosen so that
\[\mz_1-\mz = \mmu(z,\eps) = \mmu\bigl(\psi^0(\mz),\eps\bigr)= \mf^0(\mz)-\mz.\]
We therefore have
\[\psi^1\circ \mf^0(\mz) = \psi^1(\mz_1) = f_{\eps^2}(z) = f_{\eps^2}\circ \psi^0(\mz).\]
In other words,  the following diagram commutes:
\[\diagram
\mS_w \rto^{\mf^0}\dto_{\psi^0} & \C\dto^{\psi^1}\\
V_w\rto_{f_w} & \P^1(\C).
\enddiagram   \]
%

\medskip\noindent{\bf Step 2.2.} We now prove that for $(z,\eps)\in \Delta$,   the following estimate is true:
\[2\pi \i \mmu(z,\eps) = 2c_1\eps -2c_1(1-a)\eps z + \O(\eps z^2) + \O(\eps^3).\]
Indeed, observe that
\[
\frac{f_{\eps^2}(z) - \zeta_1^+(\eps)}{z-\zeta_0^+(\eps)}\cdot \frac{z-\zeta_0^-(\eps)}{f_{\eps^2}(z) - \zeta_1^-(\eps)}=  \frac{f(z) - f\bigl(\zeta_0^+(\eps)\bigr)}{z-\zeta_0^+(\eps)}\cdot \frac{z-\zeta_0^-(\eps)}{f(z) -f\bigl(\zeta_0^-(\eps)\bigr)},
\]
whence
\[2\pi \i \mmu(z,\eps) = \log\left(\frac{1-f\bigl(\zeta_0^+(\eps)\bigr)/f(z)}{1-\zeta_0^+(\eps)/z}\right)-\log\left(\frac{1-f\bigl(\zeta_0^-(\eps)\bigr)/f(z)}{1-\zeta_0^-(\eps)/z}\right).\]
Recall that $\zeta_0^\pm(\eps) = \pm c_1\eps + c_2\eps^2$, so
\[  f\bigl(\zeta_0^\pm(\eps)\bigr)= \pm c_1\eps +c_3 \eps^2 + \O(\eps^3)\with  c_3:= c_2+c_1^2.\]
Since $\mmu$ is holomorphic on $\Delta$ and $\mmu(z, 0)\equiv 0$, it admits an expansion of the form
$\mmu(z, \eps) =  u_1(z)\eps+ u_2(z)\eps^2+ \O(\eps^3)$ on $\Delta$. To find $u_1$ and $u_2$ we write
\begin{align*}
\log\left(\frac{1-f\bigl(\zeta_0^\pm(\eps)\bigr)/f(z)}{1-\zeta_0^\pm(\eps)/z}\right)&= \pm c_1\cdot\left(-\frac{1}{f(z)}+\frac{1}{z}\right)\cdot  \eps \\
&\quad +\left(-\frac{c_3}{f(z)}-\frac{c_1^2}{2\bigl(f(z)\bigr)^2}+\frac{c_2}{z}+\frac{c_1^2}{2z^2}\right)\cdot \eps^2+\O(\eps^3),
\end{align*}
and taking the difference between these two expressions we conclude that
\begin{align*}
2\pi \i \mmu(z,\eps) &= 2c_1\cdot\left(\frac{1}{z}-\frac{1}{f(z)}\right)\cdot  \eps  +\O(\eps^3) \\
&= 2c_1\eps -2c_1(1-a)\eps z + \O(\eps z^2)+\O(\eps^3) .
\end{align*}

\medskip\noindent{\bf Step 2.3.}
We finally establish    (\ref{eq:F0}).
If $\mz\in \mQ_w$ and
\[z:=\psi^0(\mz) = \i c_1\eps \cot(\pi \mz) +c_2 \eps^2 = \i c_1\eps \cot(\pi \mz) + \o(\eps),\]
  it then  follows from Step 1 of the proof (see \eqref{eq:step1}) that
$z = \O\left(|\eps|^\alpha\right)$. Since  $\alpha>1/2$ we see that   $ \O(\eps z^2) = \O\left(|\eps|^{1+2\alpha}\right)\subset \o(\eps^2)$. \footnote{This computation is responsible for the lower bound in the choice of the value of $\alpha$.}
Thus, for $\mz\in \mQ_w$,
\begin{align*}
\mmu\left(\psi^0(\mz),\eps\right) &=\ds  \frac{2c_1}{2\pi \i} \eps-\frac{c_1^2\cdot (1-a)}{\pi} \eps^2\cot(\pi \mz) + \o(\eps^2)\\
& = \ds \frac{\eps}{2} + \frac{\pi(1-a)\eps^2}{4}\cot (\pi \mz) + \o(\eps^2).
\end{align*}
Since  $\mf^0(\mz) = \mz+ \mmu\left(\psi^0(\mz),\eps\right)$, we arrive at the desired estimate \eqref{eq:F0}, and the proof of
 Step 2 is complete.

\subsubsection{Proof of Step 3}
We define $\mf^1$ on the horizontal strip $\mH$ and prove that
\[\mf^1(\mz) = \mz+\o(w).\]
The idea is simply that since the distance $\abs{\zeta_2^\pm - \zeta_1^\pm}$ is much smaller than $\abs{\zeta_1^+ - \zeta_1^-}$, $(\psi^2)^{-1}\circ \psi^1$ is very close to the identity.

\medskip\noindent{\bf Step 3.1.}
We first define $\mf^1$.
Let $\mu_1:\P^1(\C)\to \P^1(\C)$ and $\mu_2:\P^1(\C)\to \P^1(\C)$ be the M\"obius transformations defined by (recall that $\eps= \sqrt{w}$)
\[\mu_1(z):=\frac{z-\zeta_1^+(\eps)}{z-\zeta_1^-(\eps)}\and \mu_2(z):=\frac{z-\zeta_2^+(\eps)}{z-\zeta_2^-(\eps)}.\]
The M\"obius transformation $\mu:=\mu_2\circ \mu_1^{-1}:\P^1(\C)\to \P^1(\C)$
 sends $\mu_1\circ \zeta_2^+(\eps)$ to 0, $ \mu_1\circ \zeta_2^-(\eps)$ to $\infty$ and fixes 1.
Set
\[\delta^+:=\mu_1\circ \zeta_2^+(\eps)=\frac{\zeta_2^+(\eps)-\zeta_1^+(\eps)}{\zeta_2^+(\eps)-\zeta_1^-(\eps)}\and \delta^-:=\frac{1}{\mu_1\circ \zeta^-(\eps)} = \frac{\zeta_2^- (\eps)-\zeta_1^-(\eps)}{\zeta_2^-(\eps) -\zeta_1^+(\eps)}.\]
Note that
\[\zeta_2^\pm(\eps)-\zeta_1^\pm(\eps) = \O\left(\eps^4\right)\whereas \zeta_2^\pm(\eps)-\zeta_1^\mp(\eps)= \i\pi \eps\cdot \bigl(1+\o(1)\bigr),\]
therefore
\[\delta^+=\O\left(\eps^{3}\right)\and\delta^-=\O\left(\eps^{3}\right).\]
Thus, the image of the horizontal strip $\mH = \set{-1< \Im(\mz)<1}$ under the  exponential map
\[\exp:\C\ni \mz\mapsto \e^{2\pi \i \mz}\in \C\setminus \{0\}\]
avoids $\delta^+$ and $1/\delta^-$ and $\mu:\exp(\mH)\to \C\setminus \{0\}$ lifts to a map
$\mf^1:\mH\to \C$ such that the following diagram commutes:
\[\diagram
\mH \rto^{\mf^1}  \ar@/_1.5pc/[dd]_{\psi^1} \dto^\exp &  \C \ar@/^1.5pc/[dd]^{\psi^2} \dto_\exp \\
\P^1(\C) \rto^{\mu} & \P^1(\C) \\
\P^1(\C)  \rto_{{\rm id}} \uto_{\mu_1} & \P^1(\C)\uto^{\mu_2}.
\enddiagram
\]
Since $\mu(1)=1$, the choice of lift is completely determined by requiring $\mf^1(0)=0$.

\medskip\noindent{\bf Step 3.2.}
We estimate $\mf^1(\mz)-\mz$. Since $\mu(\delta^+)=0$, $\mu(1/\delta^-)=\infty$ and $\mu(1)=1$, we infer that
\[\mu(z) = z\cdot \frac{1-\delta^-}{1-\delta^+}\cdot \frac{1-\delta^+/z}{1-\delta^- z}.\]
As a consequence,
\[\mf^1(\mz)-\mz = \log(1-\delta^-) - \log(1-\delta^+) + \log (1-\delta^+/z) + \log(1-\delta^- z),\]
where   $\log$ is the principal branch of the logarithm on $\C\setminus \R^-$
(the arguments of the four logarithms are close to 1).
Since
$\delta^+=\O\left(\eps^{3}\right)$  and $\delta^-=\O\left(\eps^{3}\right),$
we conclude that
\[\sup_{\mz\in \mH} \left|\mf^1(\mz)-\mz\right|  = \O\left(\eps^{3}\right) \subset  \o(\eps^2) = \o(w),\]
as desired.

\medskip\noindent{\bf Step 3.3.} Let us show that $ \mf^1\circ \mf^0(\mQ_w)\subset \mS_{g(w)} $.
First, we saw in the course of Step 1 that  $\mQ_w\subset \mR'_w$. Now,  when $w$ is small,
\[\bigl|g(w)\bigr|^{1/4} =   \abs{w}^{1/4}+ \o\left(\abs{w}^{1/4}\right),\]
hence  $\mR'_w\subset \mS_{g(w)}$ and  the distance between the
boundaries is comparable to $|w|^{1/4}$.
Since $\mf^1\circ \mf^0(\mz)=\mz+\O\left(|w|^{1/2}\right)$ on $\mQ_w$ and $|w|^{1/2}=\o\left(|w|^{1/4}\right)$, we see that
$ \mf^1\circ \mf^0(\mQ_w)\subset \mS_{g(w)} $, as claimed.

From this we deduce  that
\[f_{w}\circ \varphi_w^{-1}(\mR_w) = f_w\circ \psi_w(\mQ_w) = \psi_{g(w)}\bigl(\mf^1\circ \mf^0(\mQ_w)\bigr)\subset \psi_{g(w)}(\mS_{g(w)}) = V_{g(w)},\]
which finishes the proof of Step 3.

\subsubsection{Proof of Step 4}We define $\mf$  on $\mR_w$ and prove that
\[\mf(\mz) = \mz+\frac{\sqrt w}{2}+\o(w).\]

Let
\[\mf:=\chi^2\circ (\mf^1\circ \mf^0)\circ (\chi^0)^{-1} : \mR_{w}\to \C,\]
so that the following diagram commutes:
\[\diagram
\mR_{w} \ar[rrrr]^{\mf}&&&&\C\\
\mQ_w   \rrto_{\mf^0} \uto^{\chi^0}&&\mf^0(\mQ_w)  \rrto_{\mf^1}&& \mS_{g(w)}.\uto_{\chi^2}
\enddiagram\]
For $\mz\in \mQ_w$,   define
\[\mv(\mz):=\frac{\sqrt w}{2} + \frac{\pi(1-a)w}{4}\cot (\pi \mz)= \frac{\sqrt w}{2}+ \o\left(|w|^{1/2}\right),\]
where the second equality follows from the fact that $\cot (\pi \mz) = \O(r_w^{-1})$ on $\mQ_w$.
Now we write
\begin{align}
\notag \mf\bigl( \chi^0(\mz)\bigr)& = \chi^2\left(\mz+ \mv(\mz)+ \o(w)\right)\\
 &=\label{eq:Fchi0} \mz+ \mv(\mz)+ \o(w) + \frac{\sqrt{g(w)}  (1-a)}{2}\cdot\log \left(\frac{2\sin\bigl(\pi \mz+\pi \mv(\mz)+\o(w)\bigr)}{\sqrt{g(w)}}\right).  \end{align}
Using
\[\sqrt{g(w)} = \sqrt{w + \O\left(w^2\right)} = \sqrt w + \O\left(|w|^{3/2}\right),\] and arguing as in Lemma  \ref{lem:chiw},
we see that the logarithm in \eqref{eq:Fchi0} is $\O(\log\abs{w})$.
Thus we infer that
\[\mf\bigl(\chi^0(\mz)\bigr) = \mz + \mv(\mz) - \frac{\sqrt w(1-a)}{2}\log  \left(\frac{2\sin\bigl(\pi \mz +\pi \mv(\mz)\bigr)}{\pi \sqrt w }\right) + \o(w),\]
as a result:
\[\mf\bigl(\chi^0(\mz)\bigr)-\chi^0(\mz) = \mv(\mz)-\frac{\sqrt w(1-a)}{2}\log\left(\frac{\sin\bigl(\pi \mz +\pi \mv(\mz)\bigr)}{\sin(\pi \mz)}\right)+\o(w).\]
From the estimate  $\mv(\mz) =\sqrt w/2 + \o\left(|w|^{1/2}\right)$, we deduce
\begin{align*}
\frac{\sin\bigl(\pi \mz +\pi \mv(\mz)\bigr)}{\sin(\pi \mz)} &= \frac{\sin(\pi \mz) +\ds  \frac{\pi \sqrt w}{2} \cos(\pi \mz) + \o\left(|w|^{1/2}\right)}{\sin(\pi \mz)} \\
&= 1 +  \frac{\pi \sqrt w}{2}\cdot \cot (\pi \mz) +\o\left(|w|^{1/2}\right).
\end{align*}
So finally,
\begin{align*}
\mf\bigl(\chi^0(\mz)\bigr)-\chi^0(\mz) &= \frac{\sqrt w}{2} + \frac{\pi(1-a)w}{4}\cot(\pi \mz) \\
&\quad - \frac{\sqrt w(1-a)}{2}\cdot \frac{\pi \sqrt w}{2}\cot (\pi \mz) + \o(w)\\
& = \frac{\sqrt w}{2}+\o(w).
\end{align*}
This completes the proof of Step 4, and accordingly of Property \ref{approx} of the approximate Fatou coordinates.
\qed

\section{Parabolic implosion}\label{sec_key}

This section   is devoted to the proof of Proposition \ref{key}.

\subsection{Set up and notation}
Let $C_f$ be a compact subset of $\bpf$ and $C_g$ be a compact subset of $\bpg$.
We need to prove that the sequence of maps
\[\C^2\ni (z,w) \mapsto \Skew^{\circ 2n+1}\bigl(z,g^{\circ n^2}(w)\bigr)\in \C^2\]
converges uniformly on $C_f\times C_g$ to the map
\[C_f\times C_g\ni (z,w)\mapsto \bigl( \L_f(z),0\bigr)\in \C\times\{0\}.\]
For $(z,w)\in C_f\times C_g$ and for $m\geq 0$, set
\[w_m:=g^{\circ m}(w).\]
This sequence converges uniformly to $0$ on $C_g$
 so the difficulty consists in proving that the first coordinate converges uniformly to $\L_f(z)$.

To do this, we will have to estimate various quantities
 which depend on an integer $k\in [0,2n+1]$ (corresponding to an iterate $m = n^2+k\in [n^2, (n+1)^2]$).
 We adopt the convention that
 the notation $\o(\cdot)$ or $\O(\cdot)$ stands for  an  estimate that
 is uniform on $C_f\times C_g$,  and  depends only on $n$, meaning that it  is  uniform with respect to $k\in [0,2n+1]$.

For $m_2\geq m_1\geq 0$, we set
\[\f_{m_2,m_1}:=f_{w_{m_2-1}}\circ \cdots \circ f_{w_{m_1}}\with f_{w}(z) :=f(z)+\frac{\pi^2}{4}w.\]
 By convention, an empty composition is the identity, whence $\f_{m,m}={\rm id}$.
Then,
\[\Skew^{\circ 2n+1}\bigl(z,g^{\circ n^2}(w)\bigr)= \bigl(\f_{(n+1)^2,n^2}(z),w_{(n+1)^2}\bigr)\]
so we must prove that
\[\f_{(n+1)^2,n^2}(z) = \L_f(z) + \o(1).\]

\subsection{Outline of the proof}

Let us recall that $R>0$ was chosen  so large  that $F:Z\mapsto -1/f(-1/Z)$ satisfies
\begin{equation}\label{eq:choiceR}
\sup_{|Z|>R} \bigl|F(Z)-Z-1\bigr|<\frac{1}{10}\and \sup_{|Z|>R} \bigl|F'(Z)-1\bigr|<\frac{1}{10}.
\end{equation}
The repelling petal $\Pr$ is the image of $-\H_R$ under the univalent map $\psi_f$, and the
notation $\psi_f^{-1}$ is reserved for the inverse branch $\psi_f^{-1}:\Pr\to \H_R$.

Set
\[k_n:=\left\lfloor n^{\alpha}\right\rfloor= \o(n)\where  \frac12 <\alpha <\frac23 \text{ is as in }\eqref{eq:alpha}.\]
The proof will be divided into four propositions that we state  independently, corresponding to three
moments of the   transition between $n^2$ and $(n+1)^2$. The proofs will be given in \S \ref{sec:eps}  to \ref{sec:leaving}.

We first show that for the first $k_n$ iterates, the orbit  stays close to an orbit of $f$  (the bound $\alpha <2/3$ is important here).

\begin{proposition}[Entering the eggbeater]\label{entering}
Assume $z\in C_f$ and let $\xniota$ be defined by $\xniota:=\f_{n^2+k_n,n^2}(z)$.\footnote{The superscript
$\iota$ stands for {\em incoming}, and in Proposition \ref{leaving} below, $o$ stands for {\em outgoing}. This convention was used in \cite{bsu}.}

Then, $\xniota\sim -1/k_n$, whence $\xniota\in \bpf$ for $n$ large enough. Moreover,
\[\phi_f\left(\xniota\right)=\phi_f\left(f^{\circ k_n}(z)\right)+\o(1)\as n\to+\infty.\]
\end{proposition}

The next two propositions concern the iterates between $n^2+k_n$ and $(n+1)^2- k_n$.

\begin{proposition}[Transition length]\label{eps}
As $n\to \infty$,
\[2n\cdot\left(\sum_{m=n^2+k_n}^{n^2+2n-k_n} \frac{\sqrt{w_m}}{2}\right)= 2n-2k_n+\o(1).\footnote{Recall that $\sqrt w$ is the square-root with positive real part.}\]
\end{proposition}

\begin{proposition}[Passing through the eggbeater]\label{passing}
Let $\left(\xniota\right)_{n\geq 0}$ be a sequence such that $\xniota\sim -1/k_n$, whence $\xniota\in \bpf$ for $n$ large enough.
Set \[\xno:=\f_{(n+1)^2-k_n,n^2+k_n}\left(\xniota\right).\]
Then, $\xno\sim 1/k_n$, whence $\xno\in \Pr$ for $n$ large enough. Moreover, as $n\cv \infty$,
\[\psi_f^{-1}\left(\xno\right) = \phi_f\left(\xniota\right) + 2n\cdot \left(\sum_{m=n^2+k_n}^{n^2+2n-k_n} \frac{\sqrt{w_m}}{2}\right) -2n + \o(1)
 = \phi_f\left(\xniota\right) -2k_n +\o(1).
 \]
\end{proposition}

The last one is similar to Proposition \ref{entering}.

\begin{proposition}[Leaving the eggbeater]\label{leaving}
Let $\left(\xno\right)_{n\geq 0}$ be a sequence contained in $\Pr$ such that  $\psi_f^{-1}\left(\xno\right) = -k_n+\O(1)$ as $n\to +\infty$. Then,
\[\f_{(n+1)^2,(n+1)^2-k_n}\left(\xno\right) = f^{\circ k_n}\left(\xno\right)+\o(1)\as n\to +\infty.\]
\end{proposition}


\medskip
\begin{proof}[Proof of Proposition \ref{key}]
We start with Proposition \ref{entering}: if $z\in C_f$, then
\[\xniota:=\f_{n^2+k_n,n^2}(z)\]
satisfies $\xniota\sim -1/k_n$ and as $n\to +\infty$,
\[\phi_f\left(\xniota\right)=\phi_f\left(f^{\circ k_n}(z)\right)+\o(1) = \phi_f(z)+k_n+\o(1).\]
According to Proposition \ref{passing},
\[\xno:=\f_{(n+1)^2-k_n,n^2+k_n}\left(\xniota\right)= \f_{(n+1)^2-k_n,n^2}(z)\]
satisfies $\xno\sim 1/k_n$ and as $n\to +\infty$,
\[\psi_f^{-1}\left(\xno\right) = \phi_f(z) +k_n -2k_n + \o(1) = \phi_f(z)-k_n+\o(1).\]
Finally, since $\phi_f(z)-k_n+\o(1) = -k_n+\O(1)$, Proposition \ref{leaving} implies that as $n\to +\infty$,
\[\f_{(n+1)^2,n^2}(z) =
\f_{(n+1)^2,(n+1)^2-k_n}\left(\xno\right)  = f^{\circ k_n}\left(\xno\right)+\o(1).\]
This in turn finishes the  proof of Proposition \ref{key}   because
\begin{align*}
f^{\circ k_n}\left(\xno\right) =  f^{\circ k_n}\circ \psi_f\circ \psi_f^{-1}\left(\xno\right)& = \psi_f\left( \psi_f^{-1}\left(\xno\right) + k_n\right)\\
&= \psi_f\bigl(\phi_f(z) + \o(1)\bigr) = \L_f(z) + \o(1).\qedhere
\end{align*}
\end{proof}

\subsection{Comparison with classical parabolic implosion}

Propositions \ref{entering}, \ref{passing} and \ref{leaving} are valid if instead of using the sequence $\bigl(w_m:=g^{\circ m}(w)\bigr)$, we use the sequence $(w'_m)$ defined by
\[w'_m:=\frac{1}{n^2}\quad\text{if}\quad n^2\leq m\leq (n+1)^2-1.\]
In that case, the only modification is for Proposition \ref{eps} which has to be replaced by:
\[2n\cdot\left(\sum_{m=n^2+k_n}^{n^2+2n-k_n} \frac{\sqrt{w'_m}}{2}\right)= 2n\cdot\left(\sum_{m=n^2+k_n}^{n^2+2n-k_n} \frac{1}{2n}\right) = 2n-2k_n+1+\o(1).\]
Following the proof of Proposition \ref{key}, we get
\[f_{1/n^2}^{\circ (2n+1)}(z) =  \psi_f\bigl(\phi_f(z) +1 + \o(1)\bigr) = f\circ \L_f(z) + \o(1).\]
We thus see that in our non-autonomous context,
 where the dynamics slowly decelerates  as the orbit transits between  the eggbeaters,
 it takes exactly one more iteration to make the  transition than in the classical case.
\subsection{Transition length\label{sec:eps}}

In this paragraph, we prove Proposition \ref{eps}, which concerns   the dynamics of $g$ only.
We need to show that
\[
\sum_{m=n^2+k_n}^{n^2+2n-k_n} \frac{\sqrt{w_m}}{2}
= 1  - \frac{k_n}{n}+ \o\lrpar{\unsur{n}}
\as n\to +\infty.\]
With $\phi_g:\bpg\to \C$ denoting  the attracting Fatou coordinate of $g$,
for all $k\geq 0$, we have
\[\phi_g(w_{n^2+k}) = \phi_g(w) + n^2+k = n^2+k + \O(1).\footnote{Recall that the notations $\o(\cdot)$ and $\O(\cdot)$ mean that the estimates are uniform on $C_f\times C_g$ and with respect to $k\in [0,2n+1]$.}
\]
As a consequence, for   $k\in [k_n,2n-k_n]$ it holds
\[w_{n^2+k} = \phi_g^{-1}\bigl(n^2+k+\O(1)\bigr) = \frac{1}{n^2+k+\O(\log n)}\]
and
\[
\sqrt{w_{n^2+k}} = \frac{1}{\sqrt{n^2+k+\O(\log n)}} = \frac{1}{n}-\frac{k}{2n^3} + \O\left(\frac{\log n}{n^3}\right).
\]
It follows that
\begin{align*}
\sum_{k=k_n}^{2n-k_n} \sqrt{w_{n^2+k}}&= \frac{2n-2k_n+1}{n}-\frac{2n(2n-2k_n+1)}{4n^3}+\O\left(\frac{\log n}{n^2}\right)
\\
&= 2-\frac{2k_n}{n}+\o\left(\frac{1}{n}\right)
\end{align*}
and we are done. \qed

%
%
%

\subsection{Entering the eggbeater\label{sec:entering}}

In this paragraph, we prove Proposition \ref{entering}, that is: if $z\in C_f$ and
$\xniota:=\f_{n^2+k_n,n^2}(z),$
then, as $n\to +\infty$,
\[\xniota\sim-\frac{1}{k_n}\and \phi_f\left(\xniota\right)=\phi_f\left(f^{\circ k_n}(z)\right)+\o(1).\]

\subsubsection{Entering the attracting petal}

Choose $\kappa_0\geq 1$ sufficiently large so that
\[f^{\circ \kappa_0}(C_f)\subset \Pa.\]
For every fixed $k\geq 0$, the sequence of polynomials $(f_{w_{n^2+k}})_{n\geq 0}$ converges locally uniformly to $f$.
It follows that for every  $k\in [1,\kappa_0]$, the sequence $\f_{n^2+k, n^2}$ converges uniformly to $f^{\circ k}$ on $C_f$.
In particular, if $n$ is large enough, then
\[
\f_{n^2+k, n^2}(C_f)\subset \bpf\text{ for }k\in [1,\kappa_0] , \text{ and }
\f_{n^2+\kappa_0, n^2}(C_f)\subset \Pa.\]
In addition, since $\f_{n^2+\kappa_0, n^2}(z)$ is close to $f^{\kappa_0}(z)$, then for large $n$   we also have
\begin{equation}\label{kappa0}
k_n>\frac{10}{\bigl|\f_{n^2+\kappa_0, n^2}(z)\bigr|}\text{ for }z\in C_f.
\end{equation}

\subsubsection{The orbit remains in the attracting petal}
We now prove that if $n$ is large enough and $k\in [\kappa_0,k_n]$, then $\f_{n^2+k, n^2}(C_f)\subset \Pa$.

For this purpose, we work in the coordinate $Z=-1/z$. For $m\geq 0$,
consider the rational map $F_m$ defined by
\[F_m(Z) := -\frac{1}{f_{w_m}(-1/Z)} = F(Z) -\frac{\pi^2w_m\cdot [F(Z)]^2}{4+\pi^2 w_m \cdot F(Z)}.\]
This has to be understood as a perturbation of $F$. Notice however that the remainder term $F_m(Z) -F(Z)$ is {\em not}
negligible with respect to $F(Z)$ as $Z\cv\infty$, so we have to control precisely for which values of $Z$ the remainder is indeed small.

Since $F(Z)\sim Z$ as $z\to \infty$ and since $w_{n^2+k}\in \O(1/n^2)$ for $k\in [0,k_n]$, we get
\[\sup_{|Z|=R} \bigl|F_{n^2+k}(Z) - F(Z)\bigr| = \o(1) \text{ and } \sup_{|Z|=2k_n} \bigl|F_{n^2+k}(Z) - F(Z)\bigr|  = \O\left(\frac{k_n^2}{n^2}\right)  =  \o(1).\]
In particular, according to the Maximum Principle and the choice of $R$ - see (\ref{eq:choiceR}) - for $n$ large enough, if $k\in [0,k_n]$ then
\[\sup_{R<|Z|<2k_n} \bigl|F_{n^2+k}(Z) - Z- 1\bigr|<\frac{1}{10}. \]

An easy induction on $k$ shows that for every $k\in [\kappa_0,k_n]$ and every  $z\in C_f$,
\begin{equation}\label{eq:Zm}
\begin{aligned}
-\frac{1}{\f_{n^2+k, n^2}(z)}
&\in \overline D\left(-\frac{1}{\f_{n^2+\kappa_0, n^2}(z)}+k-\kappa_0,\frac{k-\kappa_0}{10}\right)\\
& \quad \quad \subset \bigl\{Z\in \C~;~\Re(Z)>R\text{ and } |Z|<2k_n\bigr\}.
\end{aligned}
\end{equation}
Indeed, the induction hypothesis clearly holds for $k=\kappa_0$ and if it holds for some $k\in [\kappa_0,k_n-1]$, then
\begin{align*}
-\frac{1}{\f_{n^2+k+1, n^2}(z)}& =F_{n^2+k}\left(-\frac{1}{\f_{n^2+k, n^2}(z)}\right)
 \in \overline D\left(-\frac{1}{\f_{n^2+k, n^2}(z)}+1,\frac{1}{10}\right)\\
&\quad \quad \subset \overline D\left(-\frac{1}{\f_{n^2+\kappa_0, n^2}(z)}+k-\kappa_0+1,\frac{k-\kappa_0}{10}+\frac{1}{10}\right).
\end{align*}
If $Z$ belongs to the  latter disk, then
\[\Re(Z) > \Re\left(-\frac{1}{\f_{n^2+\kappa_0, n^2}(z)}\right)+k-\kappa_0+1-\frac{k-\kappa_0+1}{10}  > R + \frac{9}{10}(k-\kappa_0+1)>R\]
and using (\ref{kappa0}),
\[|Z| < \left|-\frac{1}{\f_{n^2+\kappa_0, n^2}(z)}\right| + k-\kappa_0+1 + \frac{k-\kappa_0+1}{10} < \frac{1}{10}k_n + \frac{11}{10} k_n <2k_n.\]
This shows that $\f_{n^2+k, n^2}(z)\in \Pa$ for all $k\in [\kappa_0,k_n]$ and all $z\in C_f$.

\subsubsection{Working in attracting Fatou coordinates}

We finally prove that for $k\in [0,k_n]$
\[\phi_f\bigl(\f_{n^2+k, n^2}(z_n)\bigr) = \phi_f\bigl(f^{\circ k}(z)\bigr)+\o(1).\]
This is clear for $k\in [1,\kappa_0]$ since for each fixed $k$, the sequence $(\f_{n^2+k, n^2})$ converges uniformly to $f^{\circ k}$ on $C_f$.
So  it is enough to prove the  estimate for $k\in [\kappa_0,k_n]$.

We have that
$\phi'_f(z)\sim1/z^2$ as $z\to 0\text{ in }\Pa.$
Also, we saw in \eqref{eq:Zm}
 that for $k\in [\kappa_0,k_n]$,
\[\left|\frac{1}{\f_{n^2+k, n^2}(z_n)}\right|  \leq 2k_n.\]
It follows that for $k\in [\kappa_0,k_n-1]$ and $z\in C_f$,
\[\sup_{I_k}|\phi'_f|\in \O(k_n^2)\with  I_k:=\bigl[f\circ \f_{n^2+k, n^2}(z),\f_{n^2+k+1, n^2}(z)\bigr],\]
whence
\begin{align*}
\phi_f\bigl(\f_{n^2+k+1, n^2}(z) \bigr)&= \phi_f\left(f\bigl(\f_{n^2+k, n^2}(z)\bigr) + \frac{\pi^2}{4} w_{n^2+k}\right) \\
&= \ds \phi_f\circ f\bigl(\f_{n^2+k, n^2}(z)\bigr) +w_{n^2+k}\cdot
\sup_{I_k}  |\phi_f'|\cdot \O(1) \\
&= \phi_f\bigl(\f_{n^2+k, n^2}(z)\bigr)+ 1 + \O\left(\frac{k_n^2}{n^2}\right).
\end{align*}
As a consequence, for $k\in [\kappa_0,k_n]$,
\begin{align*}
\phi_f\bigl(\f_{n^2+k, n^2}(z) \bigr)& =\phi_f\bigl(\f_{n^2+\kappa_0, n^2}(z) \bigr) +k-\kappa_0 + \O\left(\frac{k_n^3}{n^2}\right) \\
&= \phi_f\bigl(f^{\circ \kappa_0}(z) \bigr) + k -\kappa_0+ \o(1)\\
& = \phi_f\bigl(f^{\circ k}(z) \bigr) + \o(1),
\end{align*}
where the second equality follows from the estimate
$ \f_{n^2+\kappa_0, n^2}(z) =  f^{\circ \kappa_0}(z) + \o(1) $ and the fact that
 $k_n^3=\O\left(n^{3\alpha}\right)$ since $\alpha<2/3$.

Taking $k=k_n$, we conclude that
\[\phi_f(\xniota) = \phi_f\bigl(\f_{n^2+k_n, n^2}(z)\bigr) + \o(1) = \phi_f(z) + k_n + \o(1)= k_n+\O(1)\]
and so, $\xniota \sim -1/\phi_f(\xniota)\sim -1/k_n$
as required. The proof of Proposition \ref{entering} is completed. \qed

\subsection{Passing through the eggbeater\label{sec:passing}}

In this paragraph, we prove Proposition~\ref{passing}, that is: if $\left(\xniota\right)_{n\geq 0}$ is a sequence such that $\xniota\sim -1/k_n$ and if
 \[\xno:=\f_{(n+1)^2-k_n,n^2+k_n}\left(\xniota\right),\]
then, as $n\to +\infty$,
\[\xno\sim \frac{1}{k_n} \and \psi_f^{-1}\left(\xno\right) = \phi_f\left(\xniota\right) + n\cdot \left(\sum_{m=n^2+k_n}^{n^2+2n-k_n} \sqrt{w_m}\right) -2n + \o(1).\]
The proof relies on the formalism of  approximate Fatou coordinates introduced in Section \ref{sec:approxcoord}  and notation thereof  (in particular   Properties \ref{compatt}, \ref{comprep} and \ref{approx}).
Figure \ref{quasitrans} illustrates the proof.


\begin{figure}[htbp]
\centering
\def\svgwidth{10cm}
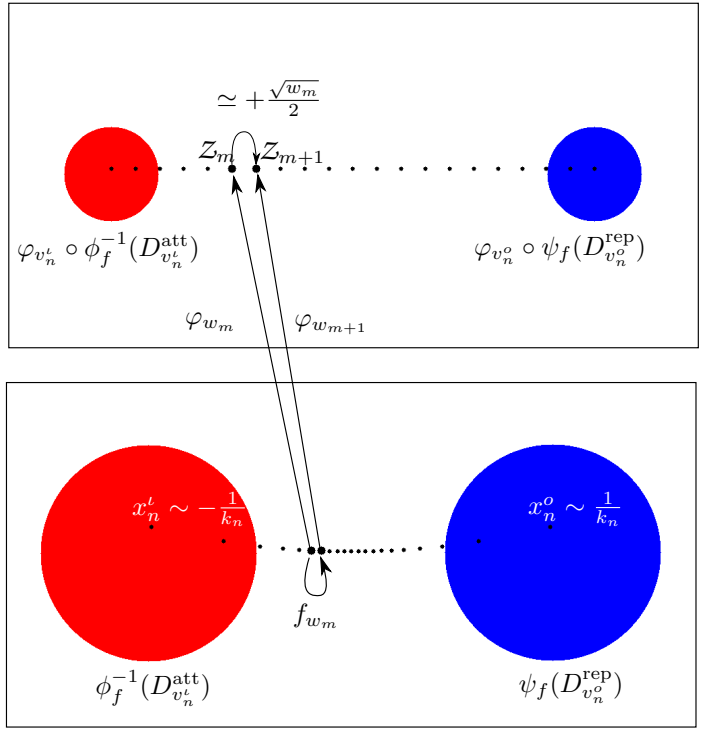
\caption{\small{ The map $\varphi_{w_{m+1}}\circ f_{w_m}\circ \varphi_{w_m}^{-1}$ is close to a translation by $\frac{\sqrt{w_m}}{2}$. }\label{quasitrans}}
\end{figure}

\begin{proof}
Let  $\uniota:=w_{n^2+k_n}$ which belongs to  $ \Pg$ for $n$ large enough.

\medskip
\noindent{\bf Step 1.}
If $\xniota\sim -1/k_n$, then for $n$ large enough, $\xniota\in \Pa$. Set $Y_n:=\phi_f(\xniota)$ and note that
\[Y_n\sim -\frac{1}{\xniota}\sim k_n\sim n^\alpha\sim  |\uniota|^{-\alpha/2}
 \whence Y_n\in D_{\uniota}^\att\for n \text{ large enough}.\]
According to Property \ref{compatt}, for $n$ large enough,  $\xniota=\phi_f^{-1}(Y_n)\in V_{\uniota}$ and
\begin{equation}\label{phiyn}
\frac{2}{\sqrt{\uniota}}\cdot \varphi_{\uniota}(\xniota) = \frac{2}{\sqrt{\uniota}}\cdot \varphi_{\uniota}\circ \phi_f^{-1}(Y_n) = Y_n+\o(1) = \phi_f(\xniota)+\o(1).
\end{equation}

\medskip
\noindent{\bf Step 2.}
We will now prove by induction on $m$ that for all $m\in  [n^2+k_n,(n+1)^2-k_n]$,
\[\f_{m,n^2+k_n}(\xniota)\in V_{w_m}\]
and
\[\mz_m:=\varphi_{w_m}\circ \f_{m,n^2+k_n}(\xniota)  = \varphi_{w_n}(\xniota) +\sum_{j=n^2+k_n}^{m-1}\left(\frac{\sqrt{w_m}}{2} +\o\left(\frac{1}{n^2}\right)  \right).\]
Indeed, for $m=n^2+k_n$, we have that $w_m=\uniota$ and according to Step 1,
\[\f_{m,n^2+k_n}(\xniota)=\xniota\in V_{\uniota}=V_{w_m},\]
so the induction hypothesis  holds in this case.

Now, assume the induction hypothesis holds for some $m\in [n^2+k_n,(n+1)^2-k_n-1]$. According to Step 1,
\begin{equation}\label{eq:varphivn}
\varphi_{\uniota}(\xniota)  =  \frac{\sqrt{\uniota}}{2}\cdot \bigl(\phi_f(\xniota)+\o(1)\bigr) = \frac{k_n}{2n}+\o\left(\frac{k_n}{n}\right).
\end{equation}
In addition,
\[
\sqrt{w_{m}}  = \frac{1}{\sqrt{n^2+\O(n)}} = \frac{1}{n}+\O\left(\frac{1}{n^2}\right).
\]
It follows that
\[\mz_{m} = \frac{k_n}{2n}+\o\left(\frac{k_n}{n}\right) + (m-n^2-k_n)\cdot \left(\frac{1}{2n} +\O\left(\frac{1}{n^2}\right) \right)= \frac{m-n^2}{2n} + \o\left(\frac{k_n}{n}\right),\]
and therefore
\[\frac{k_n}{2n} +\o\left(\frac{k_n}{n}\right) \leq\Re(\mz_{m}) \leq 1-\frac{k_n}{2n} + \o\left(\frac{k_n}{n}\right)\and \Im(\mz_m)=\o(1).\]
Since $r_{w_m} = |w_m|^{(1-\alpha)/2}\sim k_n/n$, we see that for large $n$ , $\mz_m\in \mR_{w_m}$.
According to Property \ref{approx},
\[\f_{m+1,n^2+k_n}(\xniota) = f_{w_m}\circ \f_{m,n^2+k_n}(\xniota)\in V_{w_{m+1}}\]
and
\begin{align*}
\mz_{m+1} & = \phi_{w_{m+1}}\circ f_{w_m}\circ \phi_{w_m}^{-1}(\mz_m )
= \mz_m+ \frac{\sqrt{w_m}}{2} + \o(w_m)\\
& = \varphi_{\uniota}(\xniota) +\sum_{j=n^2+k_n}^{m-1}\left(\frac{\sqrt{w_j}}{2} + \o\left(\frac{1}{n^2}\right) \right)
+  \frac{\sqrt{w_m}}{2} + \o\left(\frac{1}{n^2}\right)\\
&= \varphi_{\uniota}(\xniota) + \sum_{j=n^2+k_n}^{m}\left(\frac{\sqrt{w_j}}{2} + \o\left(\frac{1}{n^2}\right)\right).
\end{align*}

\medskip\noindent{\bf Step 3.}
We  now specialize to the case $m:=(n+1)^2-k_n$ and set
\[\uno:=w_{(n+1)^2-k_n}\and
\xno:=\f_{(n+1)^2-k_n,n^2+k_n}(\xniota).\]
According to Step 2 of the proof, $\xno\in  V_{\uno}$ and
\begin{equation}\label{eq:varphiunbis}
 \varphi_{\uno}(\xno)  = \varphi_{\uniota}(\xniota) +\sum_{j=n^2+k_n}^{n^2+2n-k_n}\left(\frac{\sqrt{w_j}}{2} +\o\left(\frac{1}{n^2}\right)  \right).
 \end{equation}
In particular, by using \eqref{eq:varphivn} and Proposition \ref{eps} we get
\begin{equation}\label{eq:varphiun}
\varphi_{\uno}(\xno)  = \frac{k_n}{2n} + \o\left(\frac{k_n}{n}\right) + 1- \frac{k_n}{n} + \o\left(\frac{1}{n}\right) = 1-\frac{k_n}{2n} + \o\left(\frac{k_n}{n}\right)
\end{equation}
Set
\[X_n := \frac{2}{\sqrt{\uno}}\cdot \bigl(\varphi_{\uno}(\xno) - 1\bigr)\sothat
\varphi_{\uno}(\xno) = 1+ \frac{\sqrt{\uno}}{2} \cdot X_n.\]
Since $2/\sqrt{\uno}\sim 2n$, from \eqref{eq:varphiun} we deduce that
$X_n= -k_n \cdot (1+\o(1))$. Since in addition
$k_n \sim (\uno)^{(1-\alpha)/2}$ it follows that for   $n$ large enough,
$X_n\in D_{\uno}^\rep$.
Thus we compute
\begin{align*}
\psi_f^{-1}(\xno)
&=\psi_f^{-1}\circ \varphi_{\uno}^{-1}\left(1+\frac{\sqrt{\uno}}{2}\cdot X_n\right)
=  X_n + \o(1)\\ 
&= \frac{2}{\sqrt{\uno}}\cdot \left( \varphi_{\uniota}(\xniota) +\left(\sum_{j=n^2+k_n}^{n^2+2n-k_n}\frac{\sqrt{w_j}}{2} \right)+\o\left(\frac{1}{n}\right)-1\right)\\
&=  \frac{2}{\sqrt{\uno}}\cdot \left( \frac{\sqrt{\uniota}}{2}\bigl( \phi_f(\xniota)+ \o(1)\bigr)
 +\left(\sum_{j=n^2+k_n}^{n^2+2n-k_n}\frac{\sqrt{w_j}}{2} \right)+\o\left(\frac{1}{n}\right)-1\right)\\
 &= \phi_f(\xniota) + 2n\cdot \left(\sum_{j=n^2+k_n}^{n^2+2n-k_n}\frac{\sqrt{w_j}}{2} \right) - 2n + \o(1)\\
\end{align*}
where we deduce the first line by Property~2, the second line follows by~\eqref{eq:varphiunbis}, the third line holds thanks to Property~1, and the final line holds since~$\uniota\sim \uno\sim \frac{1}{n^2}$. This completes the proof.
\end{proof}

\subsection{Leaving the eggbeater\label{sec:leaving}}

In this paragraph, we prove Proposition \ref{leaving}, that is: if
$\left(\xno\right)_{n\geq 0}$ is a sequence contained in the repelling petal $\Pr$ and if
\[\psi_f^{-1}\left(\xno\right) = -k_n+\O(1),\]
then, as $n\to +\infty$,
 \[\f_{(n+1)^2,(n+1)^2-k_n}\left(\xno\right)
 = f^{\circ k_n}\left(\xno\right)+\o(1).\]
Set $\xo_{n,0}:=\xno$ and for $k\in [1,k_n]$, set
\[\xo_{n,k}:= \f_{(n+1)^2-k_n+k,(n+1)^2-k_n}\left(\xno\right)\]

\subsubsection{Within the repelling petal}

Let $\kappa_1$ be an integer such that for all $n\geq 0$,
\[\Re\bigl(\psi_f^{-1}(\xno)\bigr)+k_n+R<\kappa_1.\]
We prove by induction on $k$ that for $n$ large enough, if  $k\in [0,k_n-\kappa_1]$, then
\[\xo_{n,k}\in \Pr\and  \psi_f^{-1}(\xo_{n,k})= \psi_f^{-1}(\xno)+k+k \cdot \O\left(\frac{k_n^2}{n^2}\right).\]

First, the induction hypothesis clearly holds for $k=0$. So, let us assume it is true for some $k\in [0,k_n-\kappa_1-1]$.
As in Proposition \ref{entering}, since $\alpha<2/3$ and   $k\leq k_n$, $k\O(k_n^2/n^3) \subset  \o(1).$
It follows that for large $n$,
\[\Re\bigl( \psi_f^{-1}(\xo_{n,k}) \bigr) = \Re\bigl(\psi_f^{-1}(\xno)\bigr)+k+\o(1) <k+ \kappa_1-k_n-R\leq -R-1.\]
Since $\xo_{n,k+1} = f_{w_{(n+1)^2-k_n+k}}(\xo_{n,k}) = f(\xo_{n,k})+\o(1)$ and
$f\circ \psi_f = \psi_f\circ T_1$, taking  $n$ larger if necessary,
$\xo_{n,k+1} $ belongs to the repelling petal $\Pr$.

Next, since
\[(\psi_f^{-1})'(z)\sim \frac{1}{z^2}\as z\to 0\text{ in }\Pr,\]
 as in Proposition \ref{entering} we see that
\begin{align*}
\psi_f^{-1}(\xo_{n,k+1}) &= \psi_f^{-1}\left(f(\xo_{n,k}) + \O\left(\frac{1}{n^2}\right)\right) \\
& = \psi_f^{-1}\circ f(\xo_{n,k}) +   \O\left( \frac{1}{n^2 \big|\xo_{n,k}\big|^2 }\right)  \\
&=  \psi_f^{-1}(\xo_{n,k})+1 + \O\left(\frac{k_n^2}{n^2}\right)\\
& =\psi_f^{-1}(\xno)+k+1+(k+1)\cdot \O\left(\frac{k_n^2}{n^2}\right),
\end{align*}
and the proof of the induction is completed.

\subsubsection{Leaving the repelling petal}

By the previous step we  have that
\[\psi_f^{-1}(\xo_{n,k_n-\kappa_1}) = \psi_f^{-1}(\xno) + k_n-\kappa_1 + \o(1) =  \psi_f^{-1}\circ f^{\circ k_n-\kappa_1}(\xno) + \o(1).\]
Applying $\psi_f$ on both sides yields
\[\xo_{n,k_n-\kappa_1} = f^{\circ k_n-\kappa_1}(\xno)+\o(1).\]
Since the sequence of polynomials $\f_{(n+1)^2,(n+1)^2-\kappa_1}$
converges locally uniformly to $f^{\circ \kappa_1}$, we deduce that
\[\xo_{n,k_n} = \f_{(n+1)^2,(n+1)^2-k_n}\left(\xno\right)=f^{\circ k_n}(\xno) + \o(1),\]
thereby concluding the proof of Proposition \ref{leaving}.\qed

\section{A Lavaurs map with an attracting fixed point\label{sec_lavaurs}}

This section is devoted to the proof of
 Proposition \ref{attracting}. Given $a\in \C$, let $f_a:\C\to \C$ be the cubic polynomial defined by
\[f_a(z) = z+z^2+az^3.\]
We must
 show that if $r>0$ is sufficiently close to $0$ and $a\in D(1-r,r)$, then the Lavaurs map $\L_{f_a}:\bpf\to \C$
 admits an attracting fixed point. Notice that since $\L_{f_a}$ commutes  with ${f_a}$, it therefore  has  infinitely many of them.

\begin{figure}[htbp]
\centerline{
\framebox{\includegraphics[height=8.5cm]{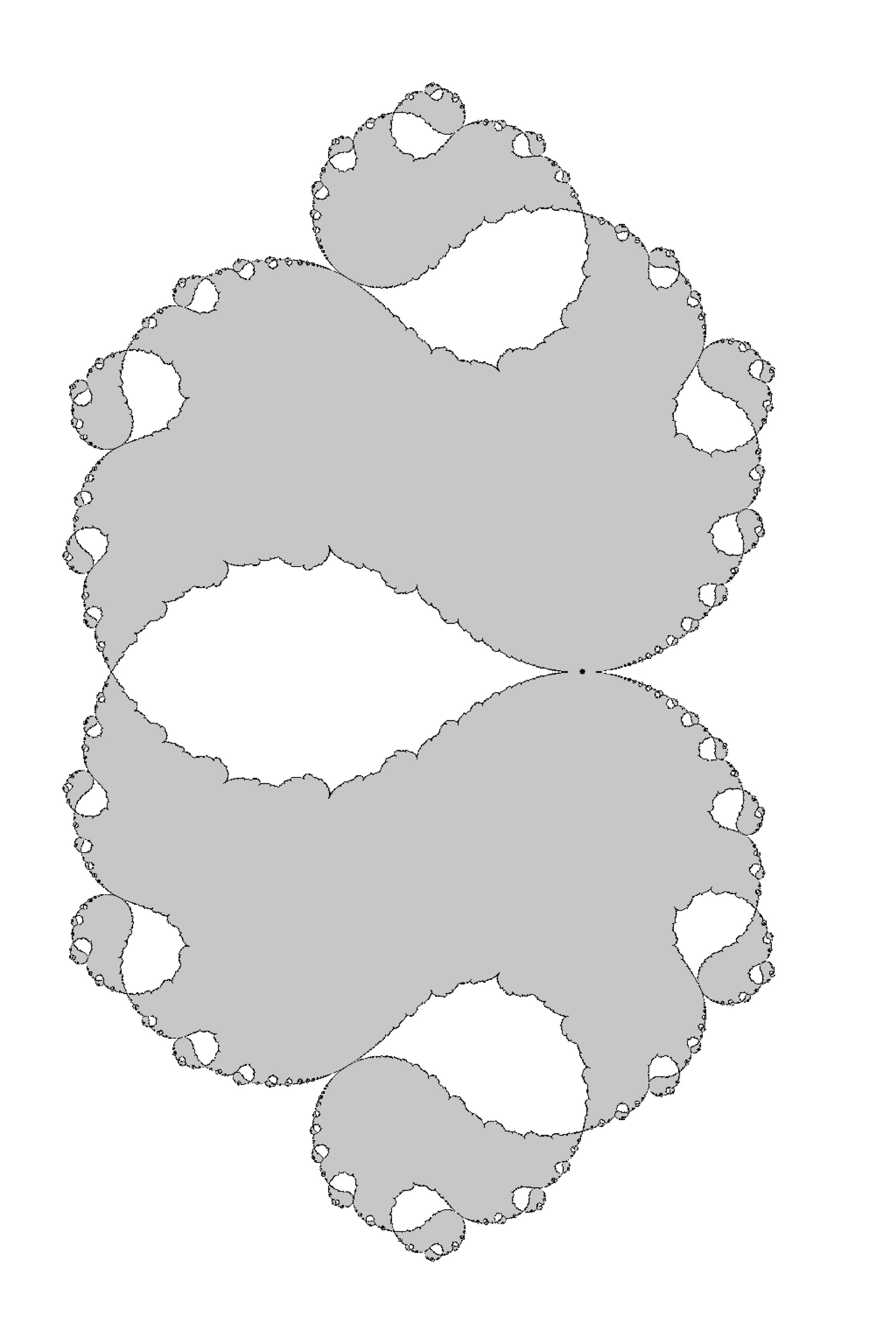}}\quad
\framebox{\includegraphics[height=8.5cm]{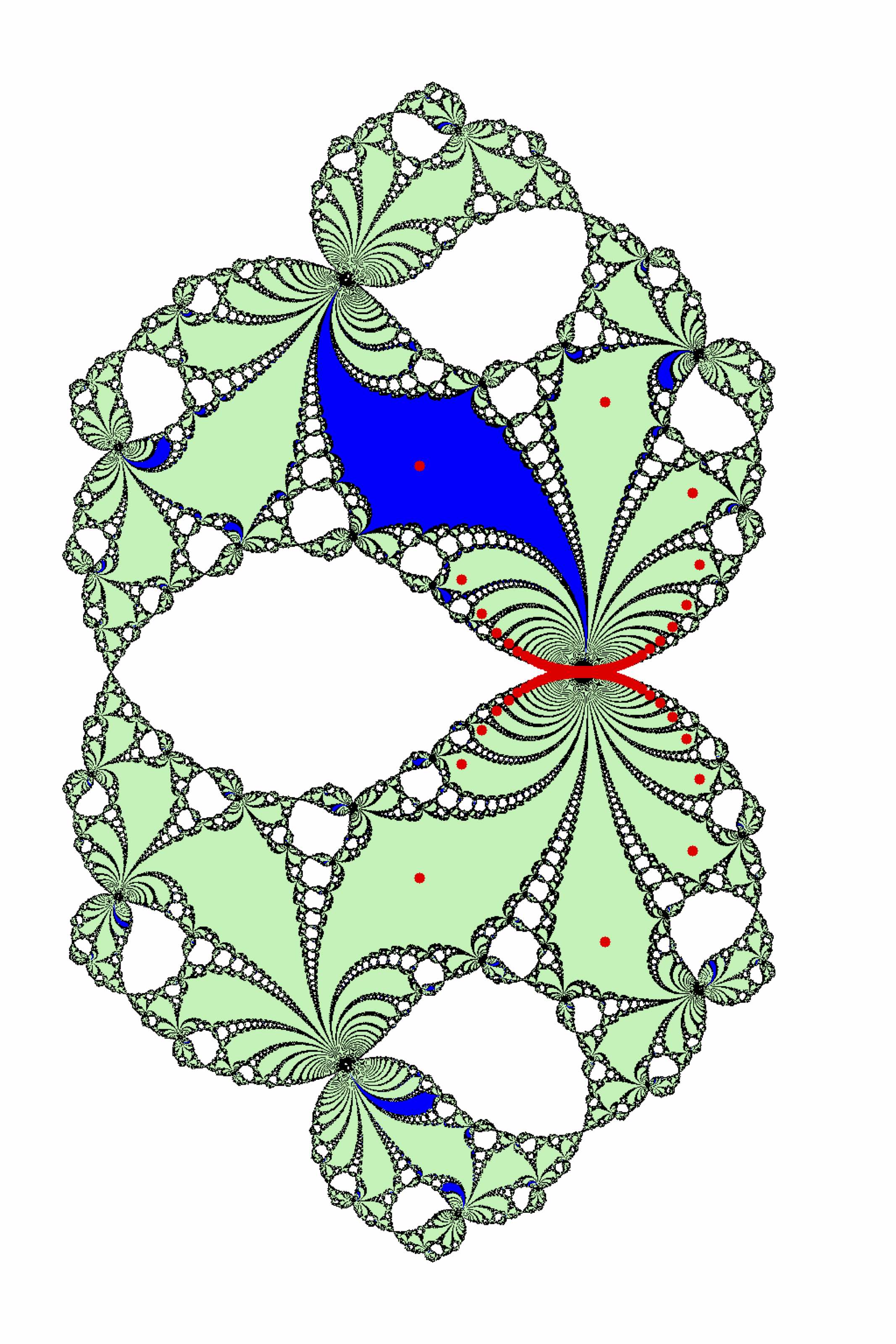}}
}
\caption{\small Behavior of $\L_f$ for $f(z) = z+ z^2 + 0.95z^3$.
Left: the set of points $z\in \bpf$ whose image by $\L_f$ remains in $\bpf$. The restriction of $\L_f$ to the bounded white domains is a covering above
$\C\setminus \overline\bpf$. Right:  the Lavaurs map $\L_f$ has  two complex conjugate sets of attracting fixed points. The fixed points of $\L_f$ are indicated as red points and their basins of attraction are colored (blue for one of the fixed points, and green for the others).\label{fig:0.95}}
\end{figure}

Set
\[{\cal U}_a:=\psi_{f_a}^{-1}({\cal B}_{f_a})\and \E_a:=\phi_{f_a}\circ \psi_{f_a}:{\cal U}_a\to \C.\]
This is an open set containing   an upper half-plane and   a lower half-plane.
Note that $\psi_{f_a}:{\cal U}_a\to {\cal B}_{f_a}$ semi-conjugates $\E_a$ to $\L_{f_a}$. Since
$\psi_{f_a}$ is univalent in a left half-plane, it is enough to show that $\E_a$ has an attracting fixed point
with real part arbitrarily close to $-\infty$. Since $\E_a$ commutes with the translation by $1$,
it is therefore enough to show that $\E_a:{\cal U}_a\to \C$ has an attracting fixed point.

The open set ${\cal U}_a$ is invariant by $T_1$ and the map  $\E_a$ commutes with $T_1$. The set
\[{\cal U}:=\bigl\{(a,Z)\in \C\times \C~;~Z\in {\cal U}_a\bigr\}\]
is an open subset of $\C^2$ and the map
\[\E:{\cal U}\ni(a,Z)\mapsto \E_a(Z)\in \C\]
is analytic. The universal cover \[\exp:{\C}\ni Z\mapsto \e^{2\pi \i Z}\in \C\setminus \{0\}\]
semi-conjugates $\E_a$ to a map
\[e_a:U_a\to \C\setminus\{0\}\with U_a:=\exp({\cal U}_a)\subset \C\setminus\{0\}.\]
The open set $U_a$ is a neighborhood of $0$ and $\infty$ in $\C\setminus \{0\}$.
The map $e_a$ has removable singularities at $0$ and $\infty$, see the proof of Lemma \ref{multhorn} below, thus it extends as a map
$e_a:\widehat U_a\to \widehat \C$, where $\Chat:=\C\cup \{\infty\}$ is the Riemann sphere and
$\widehat U_a:=U_a\cup\{0,\infty\}\subset \Chat$. We set
\[\widehat U:=\left\{(a,z)\in \C\times \Chat~;~z\in \widehat U_a\right\}.\]

\begin{lemma}\label{multhorn}
The points $0$ and $\infty$ in $\Chat$ are fixed points of $e_a:\widehat U_a\to \Chat$. Both fixed points have multiplier $\e^{2\pi ^2(1-a)}$.
\end{lemma}

\begin{proof}
As $\Im(Z)\to +\infty$, we have that
\[-\frac{1}{\psi_{f_a}(Z)} = Z + (1-a)\log(-Z)+\o(1)\]
where $\log$ is the principal branch of logarithm.
Note that $\log(-Z) = \log(Z)-\pi\i$ as $\Im(Z)\to +\infty$. Thus,
\begin{align*}
\E_a(Z) = \phi_{f_a}\circ \psi_{f_a}(Z) &= Z + (1-a)\log(-Z)+\o(1)\\
&\quad\quad\,-(1-a)\log\bigl(Z + (1-a)\log(-Z)+\o(1)\bigr) + \o(1)\\
&= Z +(1-a)\log(Z)-\pi\i(1-a)-(1-a)\log(Z) + \o(1)\\
&= Z-\pi\i(1-a)+\o(1).
\end{align*}
As a consequence, as $z=\exp(Z)\to 0$, we have that
\[e_a(z) = \e^{2\pi \i Z}\cdot \e^{2\pi^2(1-a)+\o(1)} =  e^{2\pi^2(1-a)}z\cdot \bigl(1+\o(1)\bigr),\] thus we conclude that
 $0$ is a fixed point of $e_a$ with multiplier $\e^{2\pi^2(1-a)}$. A similar argument shows
  that $\infty$ is also a fixed point of $e_a$ with multiplier $\e^{2\pi^2(1-a)}$.
\end{proof}


In particular, we see that for $a=1$, the map $e_1$ has multiple fixed points at $0$ and $\infty$.

\begin{lemma}
The multiplicity of $0$ and $\infty$ as fixed points of $e_1$ is $2$.
\end{lemma}

\begin{proof}
The mapping $e_1:\widehat U_1\to \Chat$ is a finite type analytic map in the sense of  Epstein (see Appendix \ref{sec:finitetype}). Therefore
each attracting petal at $0$ or at $\infty$ must attract the infinite orbit
of a singular value of $e_1$. Indeed if not, the component $B$ of the immediate basin containing this petal would avoid the singular values of $e_1$. The restriction $e_1:B\to B$ would then
be a covering and the corresponding attracting Fatou coordinate would extend to a covering map $\phi:B\to \C$. This would force $B$ to be isomorphic to $\C$, which is not possible since $B$ is contained in $\C\setminus \{0\}$.

\medskip

According to Proposition \ref{prop:sing}, the finite type map
 $e_1$ admits exactly two critical values (the images of the critical values of $f_1$ under the map
$\exp\circ \phi_{f_1}$) and two singular values which are respectively fixed at $0$ and $\infty$.
It follows that the number of attracting petals at $0$ plus the number of attracting petals at $\infty$
is at most $2$. So this number must be equal to $2$ and the result follows.
\end{proof}

As we perturb $a$ away from $1$, the multiple fixed point at $0$ splits into a pair of fixed points of $e_a$: one at $0$ with
multiplier $\e^{2\pi^2(1-a)}$ and another one   denoted by $\xi(a)$, with multiplier $\rho(a)$.
We  use a classical residue computation to estimate this multiplier.
Let $\gamma$ be a small loop around $0$. The Cauchy Residue Formula yields
\[\frac{1}{1-\e^{2\pi ^2(1-a)}}+ \frac{1}{1-\rho(a)} = \frac{1}{2\pi \i} \int_\gamma \frac{{\rm d} z}{z-e_a(z)} \underset{a\to 1}\longrightarrow \frac{1}{2\pi \i} \int_\gamma \frac{{\rm d} z}{z-e_1(z)}\in \C.\]
From this it follows that
\[\frac{1}{1-\rho(a)} =  \frac{1}{2\pi^2(1-a)} + \O(1)\as a\to 1.\]
Now observe that
\[\rho(a)\in \D \quad\Longleftrightarrow \quad\Re\left( \frac{1}{1-\rho(a)} \right)>\frac{1}{2},\]
and similarly,
\[a\in D(1-r,r)\quad\Longleftrightarrow \quad \Re\left(\frac{1}{2\pi ^2(1-a)}\right)>\frac{1}{4\pi^2 r}.\]
As a consequence when $r>0$ is sufficiently close to $0$ and $a\in D(1-r,1)$, we deduce
 that    $\bigl|\rho(a)\bigr|<1$, so $\xi(a)$ is an attracting fixed point.

Let finally $Z(a)$ be a preimage of $\xi(a)$ under $\exp$, that is  $\exp\bigl(Z(a)\bigr) = \xi(a)$.
We claim that for $a$ sufficiently close to $1$, $Z(a)$ is a fixed point of $\E_a$. Indeed observe first  that
$\E_a\bigl(Z(a)\bigr)-Z(a)$ is an integer which does not depend on the choice of preimage $Z(a)$.
Therefore it is sufficient to prove that
\[\lim_{a\to 1}\E_a\bigl(Z(a)\bigr)-Z(a) = 0.\]
This may be seen as follows.
The function $\E_a-\id$ is periodic of period $1$, hence of the form $u_a\circ \exp$ for some function
$u_a:\widehat U_a\to \C$. The function
\[u:\widehat U\ni (a,z)\longmapsto u_a(z)\in \C\]
is analytic,   in particular  continuous.
So,
\[\lim_{a\to 1} \E_a\bigl(Z(a)\bigr)-Z(a) = \lim_{a\to 1} u\bigl(a,\xi(a)\bigr) = u(1,0) = \lim_{\Im(Z)\to +\infty} \E_1(Z)-Z =0.\]
The last equality follows from the proof of Lemma \ref{multhorn}.
%
This shows that for $a$ sufficiently close to $1$, $Z(a)$ is a fixed point of $\E_a$ with multiplier $\rho(a)$, and the proof of
 Proposition \ref{attracting} is complete.
 \qed

\section{Wandering domains in $\R^2$}\label{sec:real}



In this section, we prove Proposition \ref{superattracting}, which shows
the existence of real polynomial maps in two complex variables with wandering Fatou components intersecting $\R^2$.
 Let us consider the   polynomial $f(z):=z+z^2+bz^4$. We seek
 a parameter $b\in (-8/27,0)$ such that the Lavaurs map
$\L_{f}$ has a fixed critical point in $\R\cap {\cal B}_{f}$.

\begin{proof}[Outline of the proof]
Set
\[b := -\frac{1+2c}{4c^3}\with c\in [-3/2,-1/2]\and f_c(z):= z+z^2-\frac{1+2c}{4c^3} z^4.\]
As $c$ increases from $-3/2$ to $-3/4$, the corresponding parameter $b$ decreases from $-4/27$ to $-8/27$ and as $c$ increases from $-3/4$ to $-1/2$, the parameter $b$ increases from $-8/27$ to $0$.  The point $c$ is a critical point of the polynomial $f_c$. As a consequence,
\[\deg_c \phi_{f_c} = \deg_cT_1\circ \phi_{f_c}  =\deg_{c} \phi_{f_c}\circ f_c = (\deg_{f_c(c)}\phi_{f_c} )\cdot (\deg_c f_c)\geq 2.\]
So, $c$ is a critical point of $\phi_{f_c}$, whence a critical point of $\L_{f_c}$.

\medskip\noindent{\bf Claim 1:} when $c\in (-3/2,-1/2]$, the critical point $c$ belongs to the parabolic basin ${\cal B}_{f_c}$.

\medskip\noindent{\bf Claim 2:} the function $\fL:(-3/2,-1/2]\to \R$ defined by $\fL(c):= \L_{f_c} (c)$ is continuous.

\medskip\noindent{\bf Claim 3:}  $\fL(-1/2)>0$.

\medskip\noindent{\bf Claim 4:} there is a sequence $c_n$ converging to $-3/2$ with $\fL(c_n)<c_n$.

\medskip
These four claims are enough to get the desired conclusion. Indeed, the function $c\mapsto \fL(c)-c$ takes a positive value at $c=-1/2$ and takes negative values arbitrarily close to $-3/2$. Since it is continuous, it follows from the Intermediate Value Theorem that it must vanish  somewhere in $(-3/2,-1/2)$.
\end{proof}

Figure \ref{graphL} shows the graph of the function $\fL:(-\frac{3}{2},-\frac{1}{2})\to \R$ which intersects the diagonal. As $c$ tends to $-\frac{3}{2}$, $\fL(c)$ accumulates the whole interval $f(\R)=(-\infty,x]$ with $x:=\frac{27}{16}+\frac{9}{8}\sqrt 3\simeq  3.63$.
\begin{figure}[htbp]
\centerline{
\framebox{\includegraphics[width=8cm]{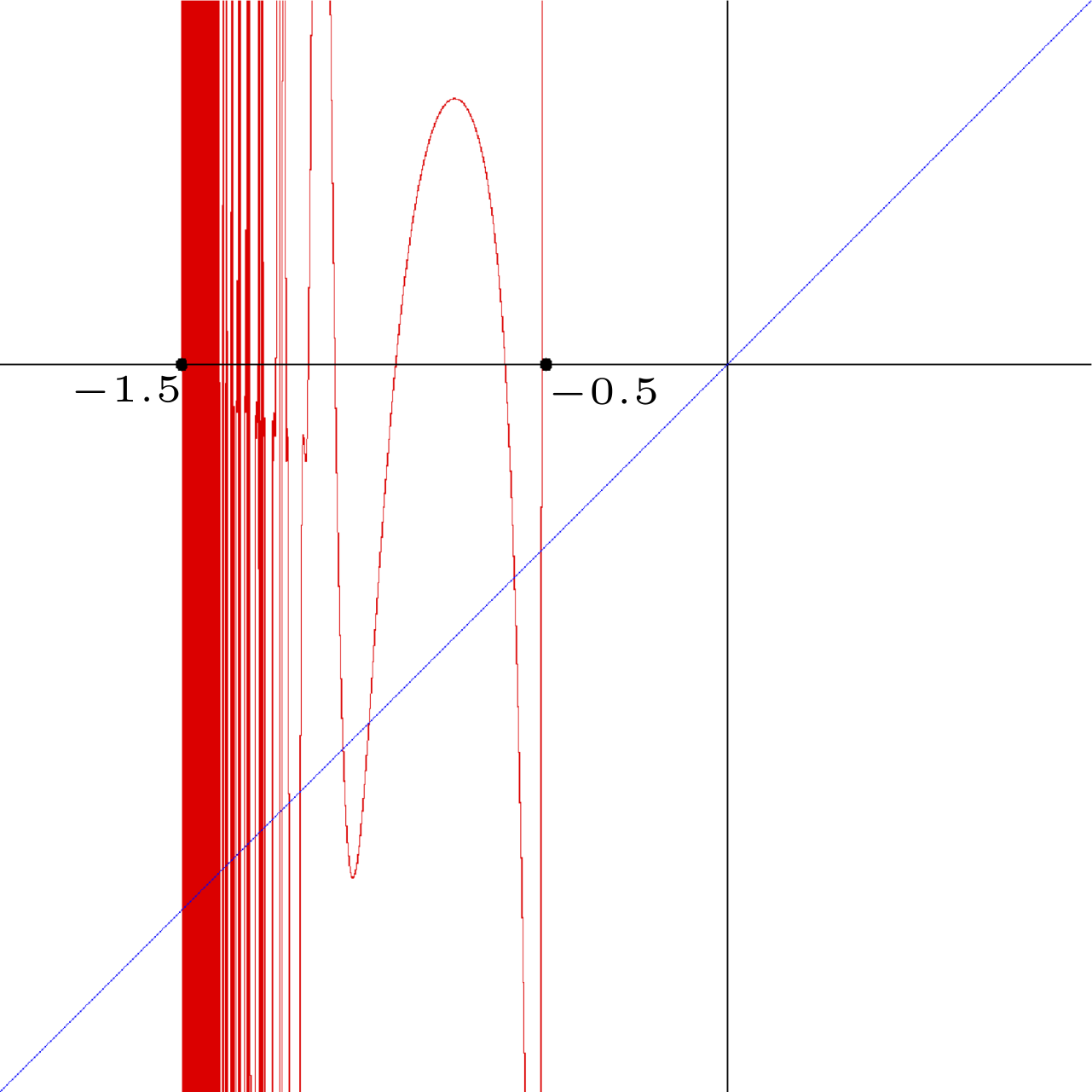}}
}
\caption{{\small The graph of the function $\fL:c\mapsto \L_{f_c}(c)$. Each intersection with the diagonal corresponds to a
 a super-attracting fixed point of $\L_{f_c}$.  \label{graphL}}}
\end{figure}

A numerical experiment suggests that the function $\fL(c)-c$ vanishes for a value of $c$ close to $-0.586$. Accordingly,  for
\[f(z) = z + z^2 -0.2136 z^4\]
the Lavaurs map $\L_f:\bpf\to \C$ has a real attracting fixed point.

%


\begin{proof}[Proof of Claim 1.]
It is enough to show that $z<f_c(z)\leq0$ for $z\in [c,0)$.
Indeed, if so, then the sequence $(f_c^n(c))_{n \geq 0}$ stays in  $[c,0)$ and it
is non-decreasing, so it must converge to the unique  fixed point in $[c,0]$,
namely to the  parabolic  fixed point $0$.

To see that $f_c - \id> 0$ on $[c,0)$, note that if $c\in [-3/2,-1/2]$, then $b\in [-8/27,0]$ and if $z\in [c,0)$, then
\[1+bz^2\geq 1+bc^2 = \frac{1}{2}-\frac{1}{4c}\geq \frac{1}{2}+\frac{1}{4}\cdot \frac{2}{3} = \frac{2}{3}.\]
Thus,
\[f_c(z) - z =  z^2 + bz^4 = z^2\cdot (1+bz^2)\geq \frac{2}{3}z^2>0.\]
%

To see that $f_c\leq 0$ on $[c,0)$, it is enough to see that $g(z):=1+z+bz^3 \geq 0$ on
$[c,0)$. As above, for $c\in [-3/2,-1/2]$ and $z\in [c,0)$, we have
\[g'(z)=1+3bz^2\geq 1+3bc^2=-\frac{1}{2}-\frac{3}{4c}\geq  -\frac{1}{2}+\frac{3}{4}\cdot \frac{2}{3} = 0.\]
Thus, $g$ is increasing on $[c,0)$ and since
\[
g(c)=1+c+bc^3=\frac{3}{4}+\frac{1}{2}c \geq 0.
\]
the proof of Claim 1 is completed.
\end{proof}

\begin{proof}[Proof of Claim 2.]
For $c\in \C\setminus{0}$ we may consider the attracting Fatou coordinate $\phi_{f_c}$ and the  repelling Fatou parameterization $
\psi_{f_c}$ of $f_c$ (normalized according to our usual convention, see Appendix \ref{app}). The formulas \eqref{eq:cvfatouatt} and
\eqref{eq:cvfatourep} defining $\phi_{f_c}$  and $\psi_{f_c}$ as limits show that $\phi_{f_c}$  and $\psi_{f_c}$ take real values on the real
axis. Define \[{\cal B}:=\bigl\{(c,z)\in \bigl(\C\setminus \{0\}\bigr)\times \C~;~z\in {\cal B}_{f_c}\bigr\}.\] Propositions \ref{prop:phicontinuous}
and \ref{prop:psicontinuous} imply that
\[{\cal B}\ni (c,z)\mapsto \phi_{f_c}(z)\in \C\and \bigl(\C\setminus \{0\}\bigr)\times \C\ni (c,Z)\mapsto \psi_{f_c}(z)\in \C\]
are continuous, as well as their composition
\[\L:{\cal B}\ni (c,z)\mapsto\psi_{f_c}\circ \phi_{f_c}(z).\]
Now for $c\in (-3/2,1/2]$, the point $(c,c)$ belongs to ${\cal B}$   so we conclude that  the function $\fL:c\mapsto \L(c,c)$ is
continuous on $(-3/2,-1/2]$.
\end{proof}

\begin{proof}[Proof of Claim 3.]
Assume $c=-1/2$ so that  $f:=f_{-1/2}$ is  the quadratic polynomial $z\mapsto z+z^2$. The repelling Fatou parametrization sends points on $\R$ which are sufficiently close to $-\infty$ to points on $\R^+$ which are close to $0$. Since
$f(\R^+)=\R^+$ and since $\psi_f(z) = f^{\circ m}\circ \psi_f(z-m)$ for all $m\geq 0$, we see that $\psi_f(\R)= \R^+$.
As a consequence, $\fL(-1/2) = \psi_f\circ \phi_f(-1/2)>0$.
\end{proof}

\begin{proof}[Proof of Claim 4.]
Let us first study the behavior of $\phi_{f_c}(c)$ when $c$ is close to $-3/2$. Putting $c=-\frac32+t$ we compute
$$f_c(c)=\frac{3}{4}c+\frac{1}{2}c^2 = -\frac{3}{4}t+\O(t^2).$$
Let $\Phi(c):=\phi_{f_c}(c)$. Then the asymptotic expansion of $\phi_{f_c}$ (see \S \ref{subs:appatt}) at 0 yields
\begin{equation*}
\Phi(c)=\phi_{f_c} \circ f_c(c) -1 = \frac{4}{3t}-\log \left( \frac{4}{3t}\right)-1+\o(1)
\end{equation*}
Thus the sequence of maps $(\Phi_n)$ defined by :
\begin{equation*}
\Phi_n(u):=\Phi \left(  -\frac{3}{2}+ \frac{4}{3(n+u)} \right) - n + \log n +1
\end{equation*}
converges uniformly to the identity  on compact intervals of $\R$.

\medskip

Now let us consider the map $f:=f_{c_0}$ for $c_0:=-3/2$.
Figure \ref{graph} shows the graph of $f$.
\begin{figure}[htbp]
\centerline{
\framebox{\includegraphics[width=8cm]{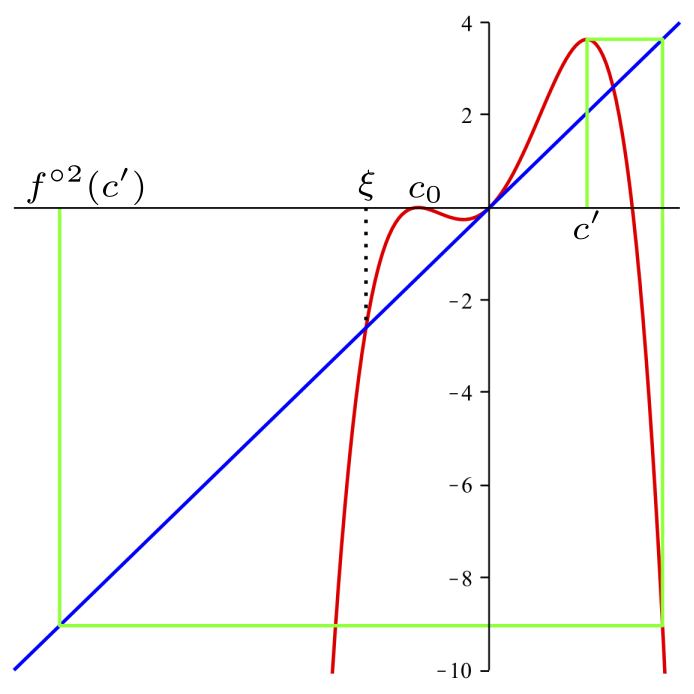}}
}
\caption{ \small The graph of the function $f(z):=z+z^2-\frac{4}{27}z^4$ and the first iterates of the critical point $c':=\frac{3}{4}(1+\sqrt 3)$. \label{graph}}
\end{figure}
The fixed points of $f$ are $0$, $\xi:=-\frac{3\sqrt 3}{2}$ and $\xi':=\frac{3\sqrt 3}{2}$. The critical points of $f$ are $c_0=-\frac{3}{2}$, $c':=\frac{3}{4}(1+\sqrt 3)$ and $c'':=\frac{3}{4}(1-\sqrt 3)$. We see that
\[f^{\circ 2}(c')<\xi<c_0<c''<0<c'<\xi'<f(c').\]
Thus, $f$ sends the interval $(-\infty,\xi)$ into itself and the orbit of any point in this interval escapes to $-\infty$.
In particular, the orbit of $c'$ escapes to $-\infty$.
In addition, $f-\id \geq 0$ on $[0,c']$, and $f$ is increasing on $[0,c']$.
So, we can define a sequence $(c'_m)_{m\geq 0}$ recursively by $c'_0:=c'$ and for $m\geq 0$,
\[c'_{m+1}\in (0,c')\and f(c'_{m+1})=c'_m.\]
This sequence is decreasing and converges to a fixed point of $f$, thus to $0$.

Choose $m_0$ large enough so that $x:=c'_{m_0}$ belongs to the repelling petal $\Pr$. Hence  $f^{\circ (m_0+2)}(x)<\xi$. Choose $\eps>0$ small enough so that for all $c\in (c_0,c_0+\eps)$, the point $x$ belongs to the repelling petal of $f_c$ and $f_c^{\circ m_0+2}(x)<\xi(c)$ where $\xi(c)$ is the leftmost fixed point of $f_c$ in $\R$. In particular, for all $m\geq m_0+2$, we have that  $f^{\circ m}(x)<\xi(c)<c$.

For $c\in (c_0,c_0+\eps)$, set
\begin{equation*}
\Psi(c):=\psi_{f_c}^{-1}(x)
\end{equation*}
and
\begin{equation*}
\Psi_n(u):=\Psi\left( -\frac{3}{2}+ \frac{4}{3(n+u)} \right).
\end{equation*}
Note that $\Psi(c)=  X_0+ \O(t)$ since $x$ lies in the repelling petal of $f_c$ and $\psi_c$
varies continuously with $c$. Therefore we also have that $\Psi_n(u)= X_0+ \O(1/n)$.

Together with the Intermediate Value Theorem, this  implies that for  large enough~$n$, the equation
\begin{equation*}
\Phi_n(u)=\Psi_n(u) + \{\log n\}
\end{equation*}
admits  at least one solution $u_n \in (X_0-1,X_0+2)$, where $\{\log n\}$ denotes  the fractional part
of $\log n$.

Now set
\[c_n := - \frac{3}{2} + \frac{4}{3(n+u_n)}\and s_n:=\lfloor \log n\rfloor.\]
We have that
\begin{align*}
\phi_{f_{c_n}}(c_n) &= \Phi_n(u_n) + n - \log n -1\\
&= \Psi_n(u_n) + \{\log n\} + n - \log n -1\\
&= \Psi_n(u_n) + n - s_n-1.
\end{align*}
Thus,
\begin{align*}
\fL(c_n)   &= \psi_{f_{c_n}} \circ \phi_{f_{c_n}}(c_n) \\
&= \psi_{f_{c_n}} \bigl(\Psi_n(u_n) + n-s_n-1\bigr)   \\
&= \psi_{f_{c_n}}\bigl( \psi_{f_{c_n}}^{-1}(x) + n - s_n-1 \bigr) = f_{c_n}^{\circ (n-s_n-1)}(x).
\end{align*}

Finally, since $n-s_n-1>m_0$ for $n$ large enough and since $f^{\circ m}(x)<c$ for all $m\geq m_0+2$ and all $c\in (c_0,c_0+\eps)$, we deduce that
$\fL(c_n)=  f_{c_n}^{\circ (n-s_n-1)}(x)<c_n$ for $n$ large enough. This completes the proof of Claim 4.
\end{proof}

\appendix

\section{Fatou coordinates\label{app}}

Throughout this Section  $f:\C\to \C$ is a polynomial of the form 
\[f(z)= z+a_2z^2+ a_3z^3+\O(z^4)\with a_2\in \C\setminus \{0\}.\]
In the coordinate $Z=-1/(a_2 z)$, the expression of $f$ becomes
\[F(Z) = Z+ 1+ \frac{b}{Z}+\O\left(\frac{1}{Z^2}\right)\with b= 1-\frac{a_3}{a_2^2}.\]
Choose $R>0$ sufficiently large so that $F$ is univalent on $\C\setminus \overline \D(0,R)$,
and the estimates \eqref{eq:choiceR} hold.
Denote by $\H_R$ the right half-plane $\H_R:=\bigl\{Z\in \C~;~\Re(Z)>R\bigr\}$ and by $-\H_R$
the
corresponding left half-plane.

Finally, denote by  $\log:\C\setminus \R^-\to \C$ be the principal branch of logarithm.

\subsection{Attracting Fatou coordinate}\label{subs:appatt}

As $\Re(Z)\to +\infty$,
\begin{align*}
F(Z) - b\log \bigl(F(Z)\bigr) &= Z+1+\frac{b}{Z} + \O\left(\frac{1}{Z^2}\right) - b\log\left(Z+1+\frac{b}{Z} + \O\left(\frac{1}{Z^2}\right) \right) \\
&=
Z-b\log Z + 1 + \O\left(\frac{1}{Z^2}\right).
\end{align*}
The  map $F$ is univalent on the right half-plane
$\H_R$ and if $Z\in \H_R$, then
 \[F^{\circ m}(Z) - b\log \bigl(F^{\circ m}(Z)\bigr) = Z-b\log Z + m + \O(1)\as m\to +\infty.\]
It follows that the sequence of univalent maps
\[\H_R\ni Z\longmapsto F^{\circ m}(Z) - m -b\log m\in \C\]
is normal and converges locally uniformly  to a univalent map $\Phi_F:\H_R\to \C$ satisfying
\[\Phi_F\circ F = T_1\circ \Phi_F\with T_1(Z) = Z+1.\]
In addition,
\[\Phi_F(Z) = Z-b\log Z + \o(1)\as \Re(Z)\to +\infty.\]
Transferring this to the initial coordinate, we see that the sequence of mappings
$\bpf\cv\C$ defined by
\begin{equation}\label{eq:cvfatouatt}
  z\longmapsto-\frac{1}{a_2\cdot f^{\circ m}(z)}-m-b\log m
\end{equation}
converges locally uniformly to an {\em attracting Fatou coordinate} $\phi_f:\bpf \to \C$ which
semi-conjugates $f:\bpf\to \bpf$ to $T_1:\C\to \C$, that is $\phi_f\circ f = T_1\circ \phi_f$, and satisfies
\[\phi_f(z) = -\frac{1}{a_2 z} -b\log \left(-\frac{1}{a_2 z}\right) + \o(1)\as \Re(-1/z)\to +\infty.\]
The restriction of $\phi_f$ to the {\em attracting petal}
\[\Pa:=\left\{z\in \C~;~\Re\left(-\frac{1}{a_2 z}\right)>R\right\}\]
coincides with $z\mapsto \Phi_F\bigl(-1/(a_2z)\bigr)$, hence it is univalent.

In addition, the convergence in \eqref{eq:cvfatouatt} is locally uniform with respect to $f$ in the open set
$\mathcal{B} = \set{(f,z), \ z\in\bpf}$, which yields the following result.

\begin{proposition}\label{prop:phicontinuous}
The map $\phi_f$ depends holomorphically on $f$.
\end{proposition}

Figure \ref{extendedattfatou} illustrates the behavior of the extended Fatou coordinate for the cubic polynomial $f_1(z) = z+ z^2+ z^3$ which has two critical points $c^\pm :=(-1\pm \i\sqrt 2)/3$. The basin of attraction $\bpf$ is colored according to the following scheme:
\begin{itemize}
\item blue if $\Im\bigl(\phi_f(z)\bigr)>\Im\bigl(\phi_f(c^+)\bigr)$,
\item red if $\Im\bigl(\phi_f(z)\bigr)<\Im\bigl(\phi_f(c^-)\bigr)$,
\item green if $\Im\bigl(\phi_f(c^-)\bigr)<\Im\bigl(\phi_f(z)\bigr)<\Im\bigl(\phi_f(c^+)\bigr)$.
\end{itemize}

\begin{figure}[htbp]
\centerline{
\framebox{\includegraphics[height=7cm]{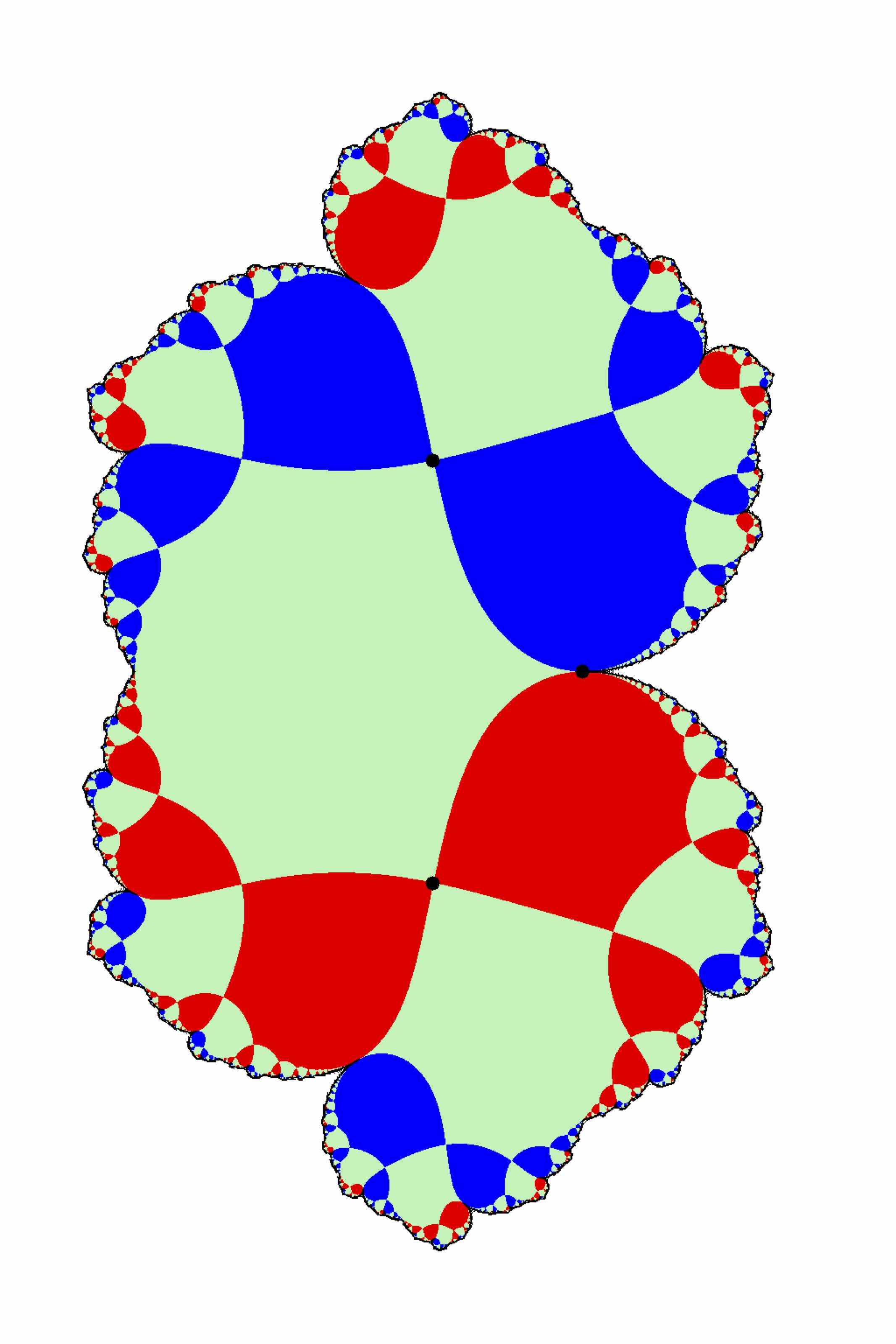}}
\quad
\framebox{\includegraphics[height=7cm]{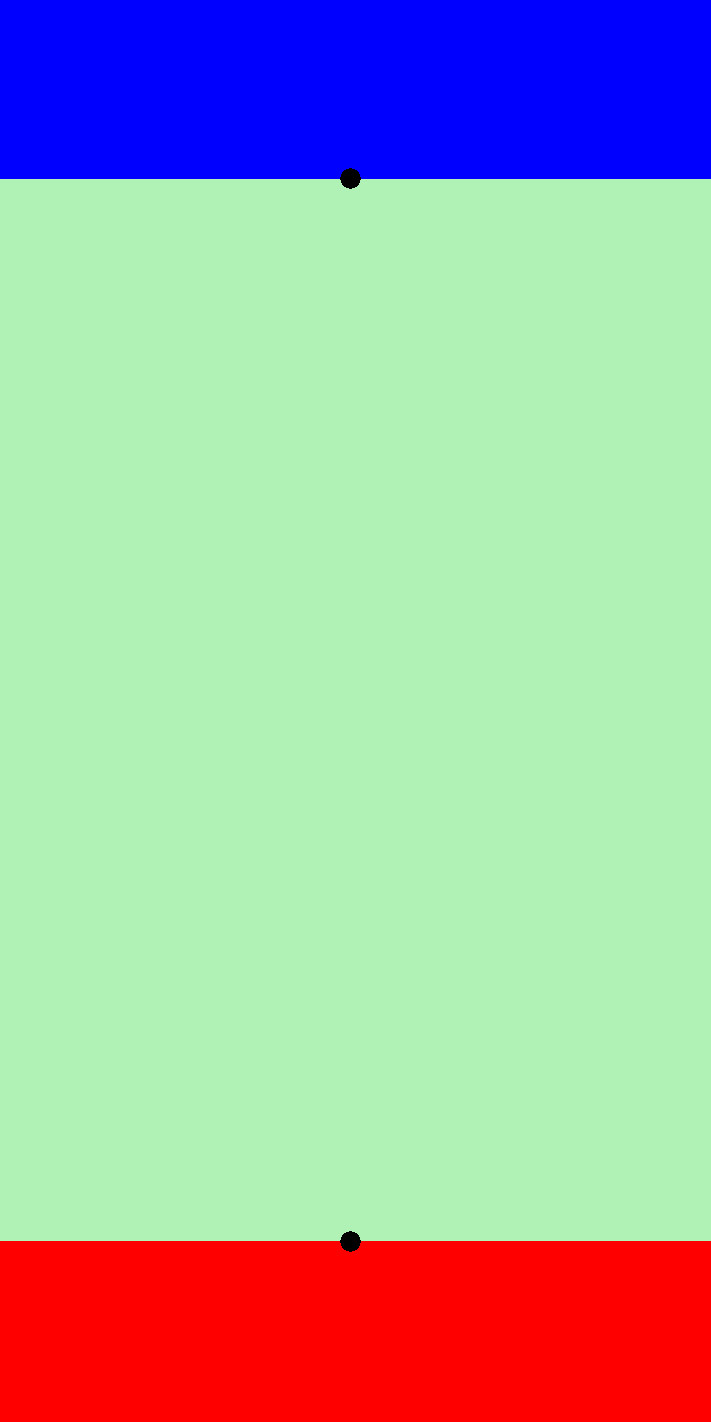}}
}
\caption{\small Left: the basin of attraction $\bpf$. The attracting Fatou coordinate $\phi_f$ is univalent on each tile. It sends each blue tile to an upper half-plane, each red tile to a lower half-plane and each green tile to an horizontal strip. The parabolic point at $0$ and the critical points $c^\pm$ are marked.  Right: the range of $\phi_f$. The points $\phi_f(c^\pm)$ are marked. \label{extendedattfatou}}
\end{figure}

\subsection{Repelling Fatou coordinate}
As $\Re(Z)\to -\infty$,
\begin{align*}
F\bigl(Z+b\log (-Z)\bigr) &= Z+b\log (-Z) + 1 + \frac{b}{Z+b\log (-Z)} + \O\left(\frac{1}{Z^2}\right)\\
& = (Z+1) + b\log(-Z-1) + \O\left(\frac{1}{Z^2}\right).
\end{align*}
It follows that if $R>0$ is sufficiently large and $\Re(Z)<-R$, then
\[F^{\circ m}\bigl((Z-m)+b\log(m-Z)\bigr) = \O(1)\as m\to +\infty.\]
In that case, the sequence of univalent maps
\[-\H_R\ni Z\longmapsto F^{\circ m}(Z-m+b\log m)\in \C \]
converges locally uniformly  to a map $\Psi_F:-\H_R\to \C$ satisfying
\[\Psi_F\circ T_1 = F \circ \Psi_F.\]
In addition,
\[\Psi_F(Z) = Z+b\log (-Z) + \o(1)\as \Re(Z)\to -\infty.\]
Transferring this to the initial coordinate, we see that the sequence of maps
\begin{equation}\label{eq:cvfatourep}
\C\ni Z\longmapsto f^{\circ m}\left(-\frac{1}{a_2 \cdot (Z-m+b\log m)}\right)\in \C
\end{equation}
converges locally uniformly to an {\em repelling Fatou parameterization} $\psi_f:\C \to \C$ which
semi-conjugates $T_1:\C\to \C$ to $f:\C\to \C$, that is $\psi_f\circ T_1 = f\circ \psi_f$, and satisfies
\[-\frac{1}{\psi_f(Z)} = Z+b\log (-Z) + \o(1)\as \Re(Z)\to -\infty.\]
The restriction of $\psi_f$ to the left half-plane $-\H_R$ coincides with $Z\mapsto -1/\bigl(a_2\Psi_F(Z)\bigr)$, whence is univalent.
The image $\Pr:=\psi_f(-\H_R)$ is called a {\em repelling petal}.

The following proposition  holds for the same reasons as in the attracting case (see \cite[Section 5]{bee} for details).

\begin{proposition}\label{prop:psicontinuous}
The map $\psi_f$ depends holomorphically on $f$.
\end{proposition}

\subsection{Lavaurs maps}

For $\sigma\in \C$, the Lavaurs map with phase $\sigma$ is the map
\[L_{f,\sigma}:=\psi_f\circ T_\sigma \circ \phi_f:\bpf\to \C\with T_\sigma(Z)=Z+\sigma.\]
In this article, we are only concerned by the Lavaurs map $\L_f:=\L_{f,0}:=\psi_f\circ \phi_f$ with phase $0$.
The relevance of Lavaurs maps is justified by the following result due to Pierre Lavaurs \cite{lavaurs}.

\begin{theorem}[Lavaurs]\label{th:lavaurs}
Let $f:\C\to \C$ be a polynomial such that $f(z) = z + z^2 + \O(z^3)$ as $z\to 0$. For $\eps\in \C$, set $f_\eps(z):=f(z)+\eps^2$.
Let $(\eps_n)_{n\geq 0}$ be a sequence of complex numbers and $(m_n)_{n\geq 0}$ be a sequence of integers, such that
\[\frac{\pi}{\eps_n}-m_n\to \sigma\in \C\as n\to +\infty.\]
Then, the sequence of polynomials $f_{\eps_n}^{\circ m_n}$ converges locally uniformly on $\bpf$ to $\L_{f,\sigma}$.
\end{theorem}

It is also relevant to consider the map
\[\E_f:=\phi_f\circ \psi_f:{\cal U}_f\to \C\with {\cal U}_f:=\psi_f^{-1}(\bpf).\]
The repelling parameterization $\psi_f$ semi-conjugates $\E_f:{\cal U}_f\to \C$ to $\L_f:\bpf\to \C$.
Figure \ref{Ef} illustrates the behavior of the map $\E_f$ for $f(z) = z+ z^2+z^3$.

\begin{figure}[htbp]
\centerline{
\framebox{\includegraphics[height=7cm]{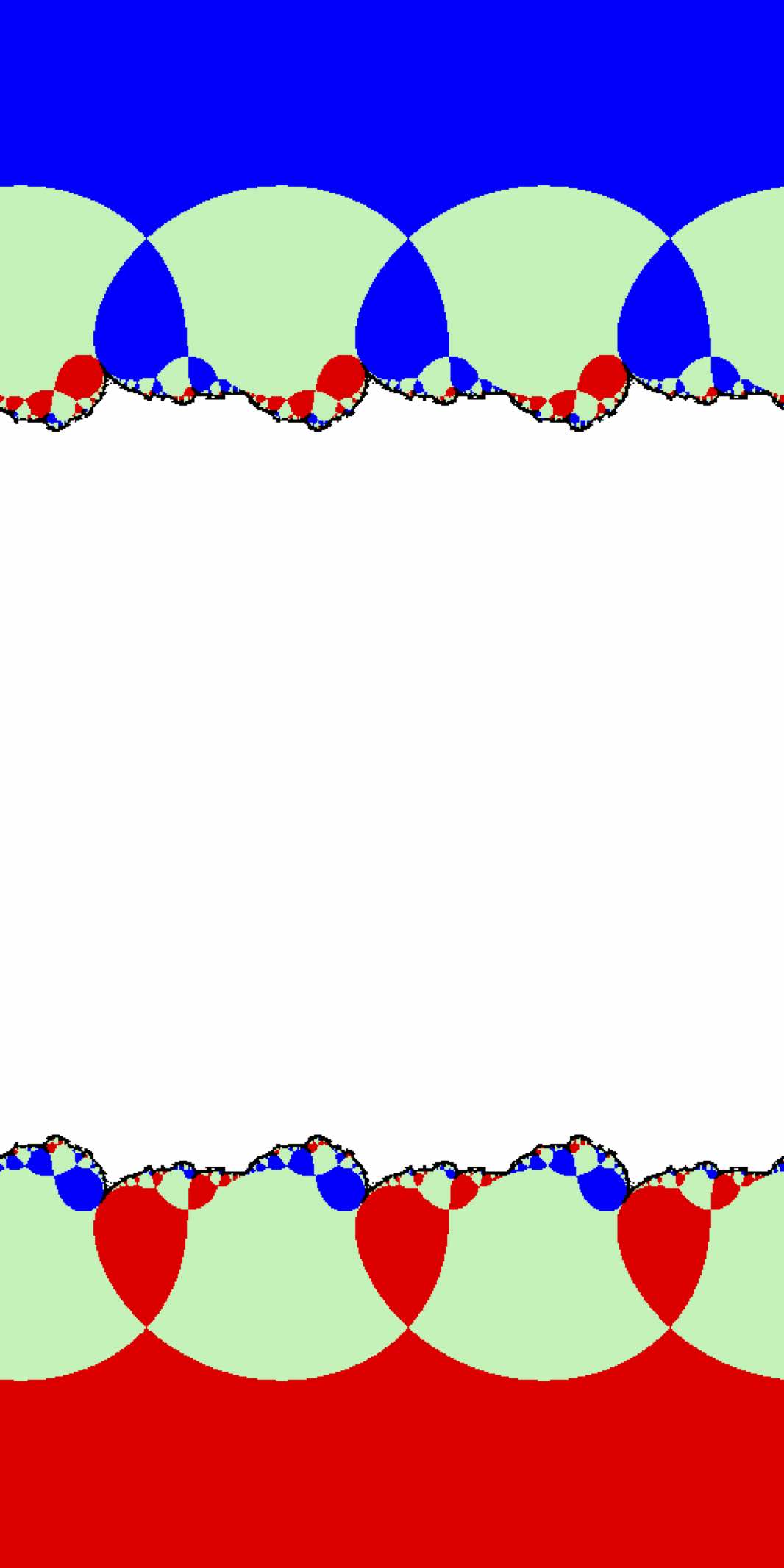}}
\qquad
\framebox{\includegraphics[height=7cm]{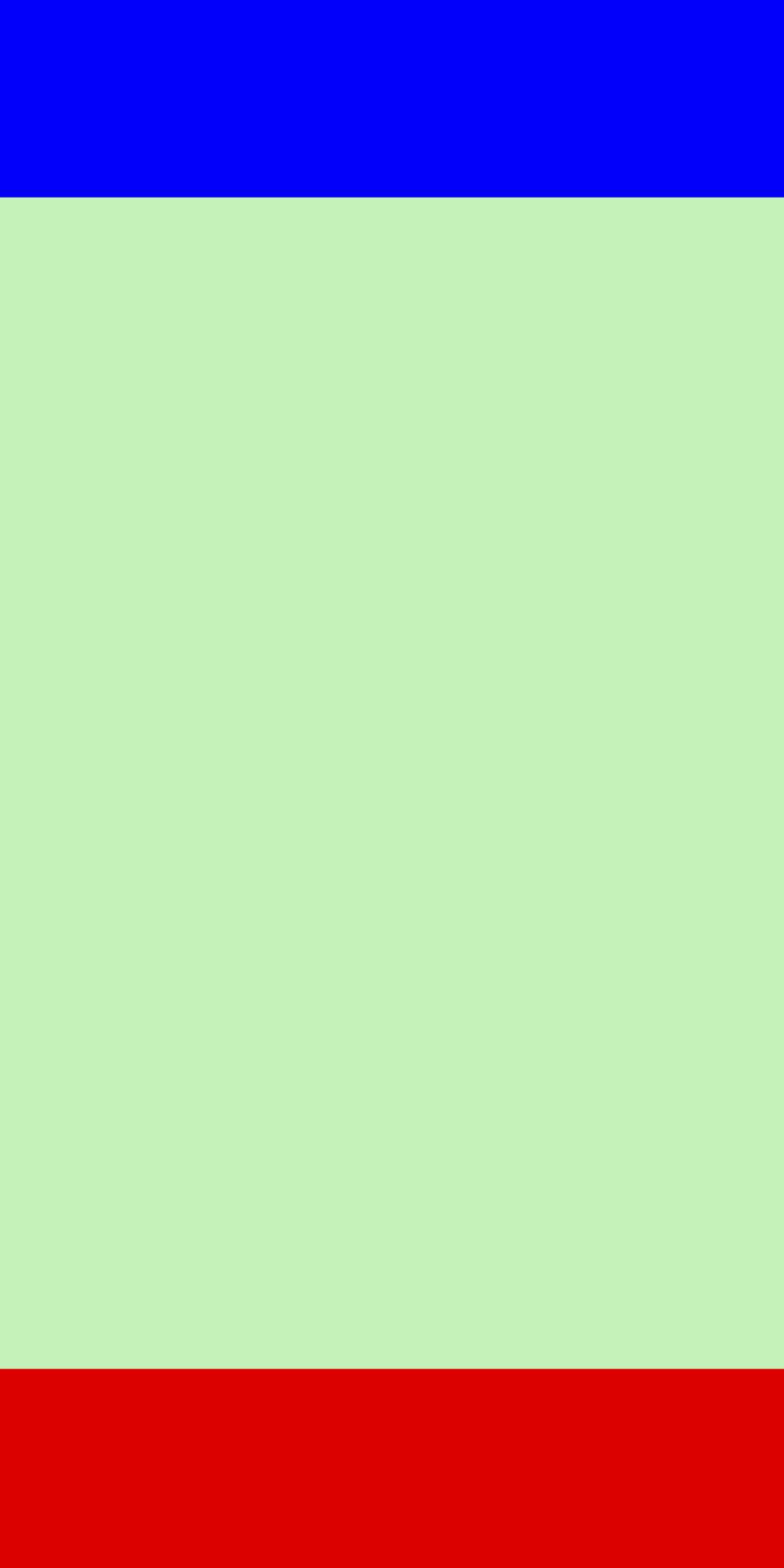}}
}
\caption{\small Behavior of the map $\E_f$ for $f(z)=z+z^2+z^3$. The domain ${\cal U}_f$ has two connected components, one containing an upper half-plane and the other containing a lower half-plane. The domain is tiled according to the behavior of $\E_f$. The restriction of $E_f$ to each tile is univalent. The image of blue tiles is the blue upper half-plane on the right. The image of red tiles is the red lower half-plane on the right. The image of green tiles is the horizontal green strip on the right. \label{Ef}}
\end{figure}

 Propositions \ref{prop:phicontinuous} and \ref{prop:psicontinuous} imply that  $\E_f$ and $\L_f$  vary nicely with $f$.

\begin{proposition}\label{prop:elcontinuous}
The mappings  $\E_f$ and $\L_f$ depend holomorphically on $f$.
\end{proposition}


Note that $\E_f$ commutes with $T_1$:
\[\E_f \circ T_1 = \phi_f\circ \psi_f \circ T_1 = \phi_f\circ f\circ \psi_f = T_1\circ \phi_f\circ \psi_f = T_1\circ \E_f.\]
So, the universal cover $\exp:\C\ni Z\longmapsto \e^{2\pi \i Z}\in \C\setminus \{0\}$ semi-conjugates $\E_f:{\cal U}_f\to \C$ to
a map
\[e_f:U_f\to \C\setminus \{0\}\with U_f:=\exp({\cal U}_f)\subset \C\setminus \{0\}.\]
The map $e_f$ has removable singularities at $0$ and $\infty$, thus it extends as a map $e_f:\widehat U_f\to \Chat$, where $\Chat:=\C\cup \{\infty\}$ is the Riemann sphere and $\widehat U_f:=U_f\cup\{0,\infty\}\subset \Chat$.
The map $e_f:\widehat U_f\to \Chat$ is called the {\em horn map} associated to $f$.
As observed by Adam Epstein in his PhD thesis \cite{e}, this horn map is a finite type analytic map (see Definition \ref{def:finite} below).

\subsection{Finite type analytic maps\label{sec:finitetype}}

Let $h:W\rightarrow X$ be an analytic map of complex 1-manifolds, possibly disconnected.
An open set $U\subseteq X$ is {\em evenly covered} by $h$ if $h_{|V}:V\rightarrow U$
is a homeomorphism for each component $V$ of $h^{-1}(U)$; we say that $x\in
X $ is a {\em regular value} for $h$ if some neighborhood $U$ of $x$ is evenly
covered, and a {\em singular value} for $h$ otherwise. Note that the set
$S(h)$ of singular values is closed. Recall that $w\in W$ is a {\em critical
point} if the derivative of $h$ at $w$ vanishes, and then $h(w)\in X$ is a
{\em critical value}. We say that $x\in X$ is an {\em asymptotic value} if $h$
approaches $x$ along some path tending to infinity relative to $W$. It
follows from elementary covering space theory that the critical values
together with the asymptotic values form a dense subset of $S(h)$. In
particular, every isolated point of $S(h)$ is a critical or asymptotic value.

\begin{definition}\label{def:finite}
An analytic map $h:W\rightarrow X$ of complex $1$-manifolds is of
{\em finite type} if
\begin{itemize}
\item $h$ is nowhere locally constant,
\item $h$ has no isolated removable singularities,
\item $X$ is a finite union of compact Riemann surfaces, and
\item $S(h)$ is finite.
\end{itemize}
\end{definition}

When $h:W\rightarrow X$ is a finite type analytic map with $W\subseteq X$, we say that $h$ is a finite type analytic map on $X$.
The reason why finite type analytic maps are relevant when studying Lavaurs maps is the following.

Let $f:\P\to \P$ be a rational map, let $\phi_f:\bpf\to \C$ be an attracting Fatou coordinate defined on the parabolic basin
of some fixed point of $f$ with multiplier $1$ and let $\psi_f:\C\to \P^1(\C)$ be a repelling Fatou parameterization associated to some fixed point of $f$ with multiplier $1$.
Define
\[\E_f=\phi_f\circ \psi_f:{\cal U}_f\to \C\with {\cal U}_f=(\psi_f)^{-1}(\bpf).\]
Finally set $\widehat U_f=\exp({\cal U}_f)\cup\{0,\infty\}$ and let $e_f:\widehat U_f\to \Chat$ be defined by
\[\exp\circ \E_f = e_f \circ \exp.\]
The following result is stated as \cite[Prop. 7.3]{bee}.

\begin{proposition}\label{prop:sing}
The map $e_f:\widehat U_f\to \Chat$ is a finite type analytic map on $\Chat$.
The singular values are:
\begin{itemize}
\item $0$ and $\infty$, which are fixed asymptotic values of $e_f$, and
\item the images by $\exp\circ \phi_f$ of the critical orbits of $f$ contained in $\bpf$, which are critical values of $e_f$.
\end{itemize}
\end{proposition}

\end{document}